\RequirePackage{fix-cm}
\documentclass[10pt, a4paper, twoside]{article}
\usepackage[left=4.6cm,right=4.6cm, top=5.65cm, bottom=5.15cm, headsep=0.4cm]{geometry}
\usepackage[utf8]{inputenc}
\usepackage{caption}
\usepackage{enumitem}
\usepackage{fancyhdr}
\AtBeginDocument{%
	\setlength{\oddsidemargin}{\dimexpr(\paperwidth-\textwidth)/2-1in}%
	\setlength{\evensidemargin}{\oddsidemargin}%
	\setlength{\topmargin}{%
		\dimexpr(\paperheight-\textheight)/2-\headheight-\headsep-1in}%
}

\makeatletter
\renewcommand\@makefntext[1]{\leftskip=2em\hskip-2em\@makefnmark#1}
\makeatother

\usepackage{graphicx}
\usepackage{mathptmx}    
\usepackage{titlesec}
\titleformat*{\section}{\bfseries}
\titleformat*{\subsection}{\mdseries}

\title{\large\bf Peak-over-Threshold Estimators for Spectral Tail Processes:\\ Random vs Deterministic Thresholds\thanks{ Holger Drees was partly supported by DFG grant DR 271/6-2 as part of the Research Unit 1735.}\footnote{\footnoterule \hspace{1.2ex} H. Drees $\cdot$ M. Kne\v{z}evi\'{c}} \footnote{Department of Mathematics, University of Hamburg, Bundesstra{\ss}e 55, 20146 Hamburg, Germany}\footnote{E-mail: holger.drees@uni-hamburg.de, miran.knezevic@uni-hamburg.de}
}
\renewcommand\footnotemark{}

\author{\normalsize  Holger Drees     \and 	 \normalsize Miran Kne\v{z}evi\'{c}}
\date{\normalsize \today}

\usepackage{amsmath,amssymb,latexsym,amsfonts}
\usepackage{dsfont}
\usepackage{float}
\usepackage[hypertexnames=false, colorlinks = true, linkcolor = blue, citecolor=blue]{hyperref}
\usepackage{natbib}
\setcitestyle{authoryear,open={(},close={)}}
\numberwithin{equation}{section}

 \newcommand{\1}{\mathds{1}}
\newcommand{\R}{\mathbb{R}}
\newcommand{\N}{\mathbb{N}}
\newcommand{\Z}{\mathbb{Z}}
\newcommand{\B}{\mathbb{B}}
\newcommand{\La}{\mathcal{L}}
\newcommand{\eps}{\varepsilon}
\newcommand{\F}{\mathcal{F}}
\newcommand{\var}{\textit{var}}

 \DeclareMathOperator{\cov}{cov}

 \newtheorem{Thm}{Theorem}{\bfseries}{\itshape}
 \newtheorem{Lem}[Thm]{Lemma}{\bfseries}{\itshape}
 \newtheorem{Prop}[Thm]{Proposition}{\bfseries}{\itshape}
 \numberwithin{Thm}{section}

 \newenvironment{proof}[1][]
{\noindent \textit{Proof{
			#1}}}{\hfill$\Box$\\}
 \newenvironment{proofofthm}[1][]
 {\noindent \textit{Proof of Theorem \ref{theo:asnormest}{
 			#1}}}{\hfill$\Box$\\}

 \newenvironment{proofoflemma}[1][]
 {\noindent \textit{Proof of Lemma \ref{lemma:OSconsist}{
			#1}}}{\hfill$\Box$\\}

 \newenvironment{proofofthmboot}[1][]
 {\noindent \textit{Proof of Theorem \ref{theo:bootstrap}{
 			#1}}}{\hfill$\Box$\\}

\allowdisplaybreaks

\begin{document}

\maketitle
\thispagestyle{empty}

\noindent
{\bfseries Abstract\ }
  The extreme value dependence of regularly varying stationary time series can be described by the spectral tail process. Drees, Segers and Warcho\l\ [Extremes 18(3): 369--402, 2015] proposed estimators of the marginal distributions of this process based on exceedances over high deterministic thresholds and analyzed their asymptotic behavior. In practice, however, versions of the estimators are applied which use exceedances over random thresholds like intermediate order statistics. We prove that these modified estimators have the same limit distributions. This finding is corroborated in a simulation study, but the version using order statistics performs a bit better for finite samples.\vspace{1ex}

\noindent
{\bfseries Keywords\ }
	Heavy tails $\cdot$ Regular variation $\cdot$ Spectral tail process $\cdot$ Stationary time series $\cdot$ Tail process $\cdot$ Threshold selection\vspace{1ex}
	
\noindent
{\bfseries Mathematics Subject Classification (2010)\ } 62G32 $\cdot$ 62M10 $\cdot$ 62G05
\vspace{2ex}

\section{Introduction}
\label{sec:intro}
\setcounter{section}{1}

By now the statistical analysis of the univariate tail behavior of stationary time series is well established.
However, in many situations understanding dependence between extreme observations is as important as the marginal extreme value behavior. For example, large losses to a financial investment pose a more serious risk if they tend to occur in clusters. Likewise, several consecutive days of high precipitation may result in a larger flooding than a single day of heavy rainfall even if the latter is more extreme.

Despite its practical importance, inference on the dependence structure between extreme observations is much less developed than marginal analysis. In the framework of regularly varying stationary time series, \cite{Basrak.2009} introduced (spectral) tail processes as a neat way to describe the dependence structure in extreme regions.

Throughout the paper, we suppose that $(X_t)_{t\in\Z}$ is a stationary real-valued time series such that all finite-dimensional marginal distributions are regularly varying. This is equivalent to the assumption that, for each $s,t\in\Z$, $s\leqslant t$, there exists a non-degenerate measure $\mu_{s,t}$ on $(\R^{t-s+1},\B^{t-s+1})$ such that
$$ \lim_{u\to\infty} \frac{P\{u^{-1}(X_s,\ldots,X_t)\in A\}}{P\{|X_0|>u\}}=\mu_{s,t}(A) $$
for all $\mu_{s,t}$-continuity sets $A$ that are bounded away from 0. \cite{Basrak.2009} proved that a stationary time series is regularly varying if and only if there exists a so-called tail process $(Y_t)_{t\in\Z}$ such that
$$ P\big( u^{-1}(X_s,\ldots,X_t)\in \cdot \mid |X_0|>u\big) \,\rightsquigarrow\, P\big\{(Y_s,\ldots, Y_t)\in\cdot\big\}
$$
as $u\to\infty$ for all $s\leqslant 0\leqslant t$, with $\rightsquigarrow$ denoting weak convergence. Then $P\{|Y_0|>x\}=x^{-\alpha}$ for all $x\geqslant 1$ and some $\alpha>0$, the index of regular variation. Moreover, the so-called spectral tail process $\Theta_t:=Y_t/|Y_0|$, $t\in\Z$,
is independent of $|Y_0|$. This process, which is also obtained as the limit of the conditional  self-normalized process
\begin{equation} \label{eq:spectralprocdef}
P\big( |X_0|^{-1}(X_s,\ldots,X_t)\in \cdot \mid |X_0|>u\big) \,\rightsquigarrow\, P\big\{(\Theta_s,\ldots, \Theta_t)\in\cdot\big\},
\end{equation}
captures the serial extreme value dependence of the time series $(X_t)_{t\in\Z}$.

Unlike $(X_t)_{t\in\Z}$, the spectral tail process is not stationary. However, it exhibits a peculiar structure which can be described by the so-called time change formula \citep[Theorem 3.1]{Basrak.2009}: for all $i,s,t\in\Z$ with $s\leqslant 0\leqslant t$ and for all measurable functions $f:\R^{t-s+1} \to \R$ satisfying $f(y_s,\ldots,y_t)=0$ whenever $y_0=0$, we have
\begin{equation}
\label{eq:tcf}
E\big[f(\Theta_{s-i},\ldots,\Theta_{t-i})\big]
=
E\Big[    f\Big(
\frac{\Theta_{s}}{|\Theta_{i}|},
\ldots,
\frac{\Theta_{t}}{|\Theta_{i}|}
\Big) \,
|\Theta_{i}|^\alpha \,
\1{\{ \Theta_i \ne 0 \}} \Big],
\end{equation}
provided the expectations exist.

If one wants to infer on the distribution of $(\Theta_t)_{t\in\Z}$ it is natural to interpret convergence \eqref{eq:spectralprocdef} as an approximation for a sufficiently high threshold $u=u_n$ and replace the unknown conditional probability by an empirical counterpart. For example, if one wants to estimate the cdf of $\Theta_t$ at $x\in\R$ for some lag $t\in\Z$, this approach leads to the so-called forward estimator
\begin{equation}\label{eq:forwarddef}
\hat{F}_{n,u_n}^{(f,\Theta_t)}(x):=\dfrac{\sum_{i=1}^n\1\{X_{i+t}/|X_i|\leqslant x,|X_i|>u_n\}}{\sum_{i=1}^n\1\{|X_i|>u_n\}}
\end{equation}
(assuming that $X_1,\ldots,X_{n+t}$ have been observed). \cite{Drees.2015} and \cite{Davis.2018} have shown  that in certain situations more efficient estimators can be constructed by using the time change formula. From \eqref{eq:tcf}, one may conclude
\begin{align}\label{eq:tcfappl}
P\{ \Theta_t \leqslant x \}
& =
\begin{cases}
1-E\big[|\Theta_{-t}|^{\alpha}
\1{\{\Theta_0/|\Theta_{-t}| > x\}}\big]
&
\text{if $x \geqslant 0$}, \\
E\big[|\Theta_{-t}|^{\alpha}\,
\1{\{\Theta_0/|\Theta_{-t}| \leqslant x\}}\big]
&
\text{if $x<0$}
\end{cases}  \nonumber \\
& = \begin{cases}
1-\lim_{u\to\infty} E\big[|X_{-t}/X_0|^{\alpha}
\1{\{X_0/|X_{-t}| > x\}} \mid |X_0|>u \big]
&
\text{if $x \geqslant 0$}, \\
\lim_{u\to\infty} E\big[|X_{-t}/X_0|^{\alpha}
\1{\{X_0/|X_{-t}| \leqslant x\}} \mid |X_0|>u \big]
&
\text{if $x<0$}
\end{cases}
\end{align}
\citep[Lemma 2.1]{Davis.2018}. Again, by interpreting the limit as an approximation and replacing the conditional expectations by empirical analogs, we obtain the backward estimator
\begin{equation}\label{eq:backwarddef}
\hat{F}_{n,u_n}^{(b,\Theta_t)}(x):=\begin{cases}
1-\dfrac{\sum_{i=1}^n|{X_{i-t}}/{X_i}|^{\hat\alpha_{n,u_n}}\1\{X_{i}/|X_{i-t}|> x,|X_i|>u_n\}}{\sum_{i=1}^n\1\{|X_i|>u_n\}} &\text{ if } x\geqslant0,\\[2.5ex]
\dfrac{\sum_{i=1}^n|{X_{i-t}}/{X_i}|^{\hat\alpha_{n,u_n}}\1\{X_{i}/|X_{i-t}|\leqslant x,|X_i|>u_n\}}{\sum_{i=1}^n\1\{|X_i|>u_n\}}& \text{ if } x<0.
\end{cases}
\end{equation}
Here $\hat\alpha_{n,u_n}$ is a suitable estimator of the index of regular variation based on the exceedances over the threshold $u_n$, e.g., the Hill-type estimator
\begin{equation} \label{eq:Hillexceeddef}
\hat\alpha_{n,u_n} := \frac{\sum_{i=1}^n \1{\{|X_i|>u_n\}}}{\sum_{i=1}^n \log(|X_i|/u_n) \1{\{|X_i|>u_n\}}}.
\end{equation}

Typically the backward estimator is more accurate than the forward estimator if $|x|$ is not too small. In particular, \cite{Drees.2015} have shown that its asymptotic variance is always smaller for $|x|\geqslant 1$ when $\alpha$ is known, the tail process is Markovian and the threshold $u_n$ is chosen as a quantile $F^\leftarrow(1-t_n)$ of the marginal cdf with $t_n\downarrow 0$ at a suitable rate. (Here and in the following $F$ denotes the cdf of $|X_0|$ and $F^\leftarrow$ its generalized inverse.) \cite{Drees.2015} and \cite{Davis.2018} also compared the performance of both estimators for finite sample sizes in a simulation study. However, while the asymptotic results have been proved when the estimators are based on exceedances over a {\em deterministic} threshold $u_n=F^\leftarrow(1-t_n)$, in the simulation study {\em empirical} quantiles, that is, order statistics, have been used. Similarly, \cite{Davis.2018} proved consistency of certain bootstrap versions of the forward and backward estimators when the thresholds $u_n$ are deterministic, but they used a version with random thresholds to construct confidence intervals in the simulations. This leaves a gap between the mathematical analysis on the one hand and the procedure commonly applied in practice on the other hand.

It is the main aim of the present paper to close this gap. While it is plausible that there is a close relationship between the performance of both versions of the aforementioned estimators (using a deterministic respectively a random threshold), it is {\em a priori} not clear whether they have the same limit distribution. In a somewhat comparable situation, \cite{Drees.2004} examined the asymptotic behavior of the maximum likelihood estimators of a scale parameter and the extreme value index in a generalized Pareto model fitted to the exceedances over a large order statistic, which had been previously studied by \cite{Smith.1987} for exceedances over a high deterministic threshold. There it turned out that using order statistics instead of true quantiles as thresholds does not influence the asymptotic behavior of the estimator of the extreme value index, but  the limit distribution of the scale estimator is different in both approaches.

In the present setting, we show in Section \ref{sec:asymp} that under suitable conditions the limit behavior of neither the forward nor the backward estimator  nor the corresponding bootstrap estimator is altered by using random  instead of deterministic thresholds. In Section \ref{sec:simus} we demonstrate in a small simulation study that while indeed the distribution of versions of the estimators using quantiles resp.\ order statistics behave similarly, the version based on exceedances over order statistics often performs slightly better for finite sample sizes. Appendix A contains tables with true values to be estimated in the simulation study. To deal with random thresholds, we must strengthen some of the conditions used by  \cite{Davis.2018}. In Appendix B, we verify these more restrictive conditions for solutions of stochastic recurrence equations. All proofs are deferred to  Appendix C. While the general approach using empirical process theory is the same as used by \cite{Davis.2018}, the proof of asymptotic equicontinuity of certain empirical processes is much more challenging in the present setting.

\section{Asymptotic results}
\label{sec:asymp}
\setcounter{section}{2}

Under suitable conditions, the joint asymptotic normality of the processes of forward and backward estimators \eqref{eq:forwarddef} and \eqref{eq:backwarddef} has been shown by \cite{Drees.2015} and \cite{Davis.2018}. However, in practice usually some data dependent threshold $\hat u_n$ is used instead of the deterministic sequence $u_n$. Here we prove that these processes converge to the same limits if the random threshold is a consistent estimator of the deterministic sequence in the sense that
\begin{equation} \label{eq:threshconsist}
S_n:=\frac{\hat u_n}{u_n} \to 1 \quad \text{in probability.}
\end{equation}
The most prominent examples are order statistics $\hat u_n=X_{n-k_n:n}$ where $u_n$ and $k_n$ are related via  $k_n=\lfloor n \bar F(u_n)\rfloor$ (with $\bar F$ denoting the survival function of $|X_1|$) and $u_n=F^\leftarrow(1-k_n/n)$.

For ease of presentation, in this section we assume that  non-negative random variables $X_{1-\tilde t},\ldots, X_{n+\tilde t}$ have been observed (i.e., $n+2\tilde t$ is the actual sample size), but our results easily carry over to real-valued observations. Here $\tilde t>0$ denotes the maximal lag we are interested in.
We define versions of the forward and backward estimators of the cdf $F^{(\Theta_t)}$ of $\Theta_t$ for $|t|\leqslant \tilde t$ based on the exceedances over $\hat u_n$ as follows:
\begin{align}
\hat{F}_{n,\hat u_n}^{(f,\Theta_t)}(x) & :=\dfrac{\sum_{i=1}^n\1\{X_{i+t}/|X_i|\leqslant x,|X_i|>\hat u_n\}}{\sum_{i=1}^n\1\{|X_i|>\hat u_n\}}, \label{eq:forranddef}\\
\hat{F}_{n,\hat u_n}^{(b,\Theta_t)}(x) & :=
1-\dfrac{\sum_{i=1}^n|{X_{i-t}}/{X_i}|^{\hat\alpha_{n,\hat u_n}}\1\{X_{i}/|X_{i-t}|> x,|X_i|>\hat u_n\}}{\sum_{i=1}^n\1\{|X_i|>\hat u_n\}}  \label{eq:backranddef}
\end{align}
where $\hat\alpha_{n,\hat u_n}$ is defined as in \eqref{eq:Hillexceeddef} with $\hat u_n$ instead of $u_n$.

For the asymptotic analysis, we follow the approach used by \cite{Davis.2018}. The forward and backward estimators can be represented in terms of certain empirical processes, so-called generalized tail array sums.  However, here we introduce an additional parameter $s$ (belonging to some neighborhood of 1) which is multiplied with the given deterministic threshold $u_n$; evaluating the processes at $s=S_n$ then leads to the estimators with random threshold.

The asymptotic behavior of such empirical processes has been studied by \cite{Drees.2010} for $\beta$-mixing time series. To be more concrete, for some $\eps>0$, let
\begin{equation} \label{eq:Xnidef}
X_{n,i} = u_n^{-1}(X_{i-\tilde t},\ldots, X_{i+\tilde t})\1{\{X_i>(1-\eps)u_n\}}, \quad 1\leqslant i\leqslant n.
\end{equation}
Then all estimators under consideration can be expressed in terms of sums of the type $\sum_{i=1}^n \psi(X_{n,i})$ for functions $\psi:[0,\infty)^{2\tilde{t}+1}\to[0,\infty)$ of the following types:
\begin{align}
\phi_{0,s}(z)&:=\log^+\Big(\frac{z_0}{s}\Big)=\log\Big(\frac{z_0}{s}\Big)\1\{z_0>s\}, \label{eq:phi0def}\\
\phi_{1,s}(z)&:=\1\{z_0>s\}, \label{eq:phi1def}\\
\phi_{2,x,s}^t(z)&:=\1\Big\{\frac{z_t}{z_0}>x,\ z_0>s\Big\}, \label{eq:phi2def}\\
\phi_{3,y,s}^t(z)&:=\Big(\frac{z_{-t}}{z_0}\Big)^{\alpha}\1\Big\{\frac{z_0}{z_{-t}}>y,\ z_{-t}>0,\ z_0>s\Big\}, \label{eq:phi3def}
\end{align}
with $z=(z_{-\tilde{t}},\dots,z_{\tilde{t}})\in[0,\infty)^{2\tilde{t}+1}$ and $\log^+ x := (\log x) \1{\{x>1\}}$. For example,\linebreak
$\hat{F}_{n,\hat u_n}^{(f,\Theta_t)}(x)=1-\sum_{i=1}^n \phi_{2,x,S_n}^t(X_{n,i})/\sum_{i=1}^n \phi_{1,S_n}^t(X_{n,i})$, provided $S_n>1-\eps$, which, according to \eqref{eq:threshconsist}, holds with probability tending to 1.

\cite{Davis.2018} established joint asymptotic normality of the suitably standardized generalized tail array sums for fixed $s=1$. The additional index $s\in[1-\eps,1+\eps]$ (for some small $\eps>0$)  can easily be incorporated in most parts of the asymptotic analysis of the generalized tail array sums. However, verifying the entropy condition needed to prove process convergence  becomes a challenging task, while it is trivial for the one parameter families of linearly ordered functions considered by \cite{Davis.2018}. To tackle this problem, we must strengthen condition (C)  of \cite{Davis.2018} and adapt some of the other conditions as follows. Let
$$  v_n:=P\{X_0>u_n\} $$
and
$$
\beta_{n,k} := \sup_{1\leqslant l\leqslant n-k-1}
E\Big[\sup_{B\in\mathcal{B}_{n,l+k+1}^n} \left\lvert P(B \mid \mathcal{B}_{n,1}^l)-P(B)\right\rvert\Big],
$$
with $\mathcal{B}_{n,i}^j$ denoting the $\sigma$-field generated by $(X_{n,l})_{i\leqslant l\leqslant j}$.
We assume that there exist sequences $l_n,r_n\to\infty$ and some $x_0\geqslant 0$ such that the following conditions hold:
\begin{description}
	\item[\bf(A($\mathbf{x_0}$))]
	The cdf of $\Theta_t$, $F^{(\Theta_t)}$, is continuous on $[x_0,\infty)$, for $|t|\in\{1,\ldots,\tilde{t}\}$.
	\item[\bf(B)]
	As $n \to \infty$, we have $l_n\to\infty$, $l_n=o(r_n)$, $r_n=o((n v_n)^{1/2})$, $r_nv_n\to 0$, and \\ $\beta_{n,l_n}n/r_n\to 0$.
\end{description}
Without condition (A($x_0$)) one cannot expect uniform convergence of the estimators.
Condition~(B) imposes restrictions on the rate at which $v_n$ tends to 0 and thus on the rate at which $u_n$ tends to $\infty$. Often, the $\beta$-mixing coefficients decay geometrically, i.e., $\beta_{n,k}=O(\eta^k)$ for some $\eta\in(0,1)$. Then one may choose $l_n=O(\log n)$, and Condition~(B) is fulfilled for a suitably chosen $r_n$ if $(\log n)^2/n=o(v_n)$ and $v_n =o(1/(\log n))$.
\begin{description}
	\item[\bf(C)]
	For all $0\leqslant j\leqslant k\leqslant r_n$, there exist
	\begin{align}
	\label{eq:snkdef}
	s_n(k)& \geqslant E\Big[\max\Big\{\log\Big(\frac{X_0}{(1-\eps)u_n}\Big),\1\{X_0>(1-\varepsilon)u_n\Big\} \nonumber\\
	& \hspace*{0.45cm} \times\max\Big\{\log\Big(\frac{X_k}{(1-\eps)u_n}\Big),\1\{X_k>(1-\eps)u_n\}\Big\} \,\Big|\, X_0>(1-\eps)u_n\Big]\\
	\tilde s_n(j,&k)  \geqslant P\big(X_j>(1-\eps)u_n, X_k>(1-\eps)u_n \mid X_0>(1-\eps)u_n\big) \label{eq:snjkdef}
	\end{align}
	such that $s_\infty(k)=\lim_{n\to\infty}s_n(k)$  and $\tilde s_\infty(j,k)=\lim_{n\to\infty}\tilde s_n(j,k)$ exist, and \linebreak $\lim_{n\to\infty} \sum_{k=1}^{r_n} s_n(k) = \sum_{k=1}^\infty s_\infty(k)<\infty$ and $\lim_{n\to\infty} \sum_{1\leqslant j\leqslant k\leqslant r_n} \tilde s_n(j,k) = $ \linebreak $ \sum_{1\leqslant j\leqslant k<\infty } \tilde s_\infty(j,k)<\infty$ hold.
	
	Moreover,  there exists $\delta>0$ such that
	\begin{align}
	\sum_{k=1}^{r_n}
	\Big(
	E\Big[
	\Bigl( \log^+ \Bigl( \frac{X_0}{(1-\eps)u_n}\Big)\log^+\Big(\frac{X_k}{(1-\eps)u_n} & \Bigr) \Bigr)^{1+\delta}
	\, \Big| \,
	X_0 > (1-\eps) u_n
	\Big]
	\Big)^{1/(1+\delta)} \nonumber\\
	& = O(1),
	\qquad n \to \infty. \label{eq:psibdd}
	\end{align}
\end{description}
Note that condition (C) could be stated with $s_n(k)$ and $\tilde{s}_n(j,k)$ equal to the corresponding conditional probability resp.\ expectation. Then the existence of limits $s_{\infty}(k)$ and $\tilde{s}_{\infty}(j,k)$  is guaranteed by regular variation. However, it is often difficult to prove that the sums over $s_n(k)$, resp.\ $\tilde{s}_n(j,k)$ converge to the corresponding sums of these limits. In contrast, it may be quite easy to bound the conditional probabilities and expectations by simple expressions and prove convergence of the resulting sums.
\cite{Drees.2015} and \cite{Davis.2018} discussed techniques to verify weaker versions of the conditions \eqref{eq:snkdef} and \eqref{eq:psibdd} for specific time series models. We prove in Appendix B that the more restrictive condition (C) is fulfilled by solutions to stochastic recurrence equations under mild assumptions.

These conditions suffice to prove convergence of the processes of generalized tail array sums centered by their respective means. 
To replace these means by their limits in terms of the spectral process, we need additional conditions which ensure that the bias of the forward, the backward and the Hill type estimators are asymptotically negligible: for all $|t|\in\{1,\ldots,\tilde t\}$  and all sequences $s_n\to 1$ one has
\begin{align}
\sup_{x\in[x_0,\infty)}&\Big|P\Big(\frac{X_t}{X_0}\leqslant x\ \Big|\ X_0>s_nu_n\Big)-F^{(\Theta_t)}(x)\Big|=o((nv_n)^{-1/2}), \label{eq:biasforcond}\\
\sup_{y\in[y_0,\infty)}&\Big|1-E\Big[\Big(\frac{X_{-t}}{X_0}\Big)^{\alpha}\1\{X_0/X_{-t}>y\}\ \Big|\ X_0>s_nu_n\Big]-{F}^{(\Theta_t)}(y)\Big|=o((nv_n)^{-1/2}), \label{eq:biasbackcond}\\
&\Big|E[\log(X_0/(s_nu_n))\ |\ X_0>s_nu_n]-1/\alpha\Big|=o((nv_n)^{-1/2}). \label{eq:biasHillcond}
\end{align}
Note that the convergence of the left hand sides to 0 follows from our basic assumptions. Here, we require that this convergence is sufficiently fast which usually implies an upper bound on the rate at which $nv_n$ tends to $\infty$ (or, equivalently, that the threshold $u_n$ tends to $\infty$ sufficiently fast).

\begin{Thm}
	\label{theo:asnormest}
	Let $(X_t)_{t\in\Z}$ be a stationary, regularly varying time series. If the conditions \eqref{eq:threshconsist}, (A($x_0$)), (B), (C) and \eqref{eq:biasforcond}--\eqref{eq:biasHillcond} are fulfilled for some $x_0\geqslant 0$ and some $y_0\in [x_0,\infty)\cap (0,\infty)$, then
	\begin{multline} \label{eq:weakconv}
	(nv_n)^{1/2}
	\begin{pmatrix}
	\big(\hat{F}_{n,\hat u_n}^{(f,\Theta_t)}(x_t)- F^{(\Theta_t)}(x_t)\big)_{x_t\in[x_0,\infty)} \\
	\big(\hat{F}_{n,\hat u_n}^{(b,\Theta_t)}(y_t)- F^{(\Theta_t)}(y_t)\big)_{y_t\in[y_0,\infty)}
	\end{pmatrix}_{|t|\in\{1,\ldots,\tilde{t}\}} \\
	\rightsquigarrow
	\begin{pmatrix}
	( Z(\phi_{2,x_t,1}^t)-\bar F^{(\Theta_t)}(x_t) Z(\phi_{1,1}))_{x_t\in[x_0,\infty)} \\
	\big(Z(\phi_{3,y_t,1}^t)-\bar F^{(\Theta_t)}(y_t) Z(\phi_{1,1})+Z_{\alpha}(y_t)\big)_{y_t\in[y_0,\infty)}
	\end{pmatrix}_{|t|\in\{1,\ldots,\tilde{t}\}}
	\end{multline}
	with
	\begin{equation*}
	Z_{\alpha}(y_t)=(\alpha^2 Z(\phi_{0,1})-\alpha Z(\phi_{1,1}))E\big[(\log \Theta_t) \, \1\{\Theta_t>y_t\}\big],
	\end{equation*}
	where $Z$ is the centered Gaussian process, indexed by functions defined in \eqref{eq:phi0def}--\eqref{eq:phi3def}, whose covariance function is given in \eqref{eq:cov_emp_pr}, and $\bar F^{(\Theta_t)}:=1-F^{(\Theta_t)}$ denotes the survival function of $\Theta_t$.
\end{Thm}
A detailed representation of the relevant covariances is given in the Supplement.

The limit process is exactly the same as for the forward and backward estimator based on the exceedances over the deterministic threshold $u_n$, derived by \cite{Davis.2018}.  So, retrospectively, it is justified that \cite{Davis.2018} used order statistics $X_{n-k_n:n}$ as thresholds $\hat u_n$ instead of $u_n=F^\leftarrow(1-k_n/n)$, since the following lemma shows that,  under the conditions used in Theorem \ref{theo:asnormest}, order statistics indeed fulfill \eqref{eq:threshconsist}.
\begin{Lem} \label{lemma:OSconsist}
	Let $(k_n)_{n\in\N}$ be an intermediate sequence, that is $k_n\to\infty$ and $k_n/n\to 0$.
	If condition (B) holds and there are $s_n(k)\geqslant P\big(X_k>(1-\eps)u_n\mid X_0>(1-\eps)u_n\big)$ such that $s_\infty(k)=\lim_{n\to\infty}s_n(k)$ exists, and $\lim_{n\to\infty} \sum_{k=1}^{r_n} s_n(k) = \sum_{k=1}^\infty s_\infty(k)<\infty$,
	then $X_{n-k_n:n}/F^\leftarrow(1-k_n/n)\to 1$ in probability.
\end{Lem}
Note that the existence of the constants $s_n(k)$ follows from assumption \eqref{eq:snkdef} in condition (C).
The consistency of intermediate order statistics is also an immediate consequence of Theorem 2.1 of \cite{Drees.2003}, which has been proved under somewhat different conditions.

\subsection*{Multiplier block bootstrap}

The covariance function of the limit process in  Theorem \ref{theo:asnormest} is too complex to be directly used for the construction of confidence regions for $F^{(\Theta_t)}(x)$. Therefore, we resort to a resampling method.
\cite{Davis.2018} proposed two bootstrap schemes: the stationary block bootstrap as used in \cite{Davis.2012} and the multiplier block bootstrap as used in \cite{Drees.2015b}. The multiplier block bootstrap versions of the forward and backward estimator are given by
\begin{align*} 
\hat{F}_{n,u_n}^{*(f,\Theta_t)}(x)&:=\frac{\sum_{j=1}^{m_n}(1+\xi_j)\sum_{i\in I_j}\1\{X_{i+t}/|X_i|\leqslant x, |X_i|>{u}_n\}}{\sum_{j=1}^{m_n}(1+\xi_j)\sum_{i\in I_j}\1\{|X_i|>{u}_n\}},\\
\hat{F}_{n,u_n}^{*(b,\Theta_t)}(x)&:=
1-\dfrac{\sum_{j=1}^{m_n}(1+\xi_j)\sum_{i\in I_j}|\frac{X_{i-t}}{X_i}|^{\hat{\alpha}_{n,u_n}^*}\1\{X_{i}/|X_{i-t}|> x,|X_i|>{u}_n\}}{\sum_{j=1}^{m_n}(1+\xi_j)\sum_{i\in I_j}\1\{|X_i|>{u}_n\}}
\end{align*}
for $x\geqslant 0$, with
\begin{align*}
\hat{\alpha}_{n,u_n}^*:=\frac{\sum_{j=1}^{m_n}(1+\xi_j)\sum_{i\in I_j}\1\{|X_i|>{u}_n\}}{\sum_{j=1}^{m_n}(1+\xi_j)\sum_{i\in I_j}\log(|X_i|/{u}_n)\1\{|X_i|>{u}_n\}},
\end{align*}
where $\xi_j$, $j\in\N$, are (bounded) iid random variables independent of $(X_t)_{t\in\Z}$ with $E[\xi_j]=0$ and $\var(\xi_j)=1$, and $I_j=\{(j-1)r_n+1,\dots,jr_n\}$. \cite{Davis.2018} proved that these multiplier block bootstrap versions are consistent when the thresholds $u_n$ are deterministic and suitable chosen. (The  validity of the stationary block bootstrap versions of the forward and backward estimator was not established.)
In simulations, both bootstrap methods were applied with order statistics as threshold. The approximation of the error distribution obtained by multiplier block bootstrap turned out to be much more accurate than that of the stationary bootstrap. Here we will give an asymptotic justification for using multiplier block bootstrap with random  instead of deterministic thresholds.

For the sake of brevity, we focus on estimators of $F^{(\Theta_t)}(x)$ for a fixed $x\geqslant x_0$. By $P_{\xi}$ we denote probabilities w.r.t.\ $\xi=(\xi_j)_{j\in\N}$, i.e., conditional probabilities given $(X_{n,i})_{1\leqslant i\leqslant n}$.
\begin{Thm}\label{theo:bootstrap}
	Let $\xi_j$, $j\in\N$, be bounded iid random variables with $E[\xi_j]=0$ and $\var(\xi_j)=1$ independent of $(X_t)_{t\in\Z}$. Then, under the conditions of Theorem \ref{theo:asnormest}, for all $x\geqslant x_0$, $y\geqslant y_0$,
	\begin{align*}
	\sup_{r,s\in\R^{2\tilde{t}}}&\Big| P_{\xi}\Big[ (nv_{n})^{1/2}\big(\hat{F}_{n,\hat{u}_n}^{*(f,\Theta_t)}(x)-\hat{F}_{n,\hat{u}_n}^{(f,\Theta_t)}(x)\big)\leqslant r_t,\\
	&\hspace{8ex} (nv_{n})^{1/2}\big(\hat{F}_{n,\hat{u}_n}^{*(b,\Theta_t)}(y)-\hat{F}_{n,\hat{u}_n}^{(b,\Theta_t)}(y)\big)\leqslant s_t,\forall|t|\in\{1,\dots,\tilde{t}\}\Big]\\
	&-P_{\xi}\Big[ (nv_{n})^{1/2}\big(\hat{F}_{n,\hat{u}_n}^{(f,\Theta_t)}(x)-F^{(\Theta_t)}(x)\big)\leqslant r_t,\\
	&\hspace{8ex} (nv_{n})^{1/2}\big(\hat{F}_{n,\hat{u}_n}^{(b,\Theta_t)}(y)-F^{(\Theta_t)}(y)\big)\leqslant s_t,\forall|t|\in\{1,\dots,\tilde{t}\}\Big]\Big| \to 0
	\end{align*}
	in probability.
\end{Thm}

This result shows that the approach to constructing confidence intervals used in the simulations of \cite{Davis.2018} is mathematically sound.

\section{Simulations}
\label{sec:simus}

In this section, we compare the finite-sample behavior of the estimators of the survival function of the tail spectral process at lag $t$ based on the exceedances over deterministic thresholds and exceedances over the corresponding order statistics, respectively. Specifically, we examine the forward estimator (SFE) $1-\hat{F}_{n,u}^{(f,\Theta_t)}(x)$ defined in \eqref{eq:forwarddef}
and the backward estimator (SBE) $1-\hat{F}_{n,u}^{(b,\Theta_t)}(x)$  of $P\{\Theta_t>x\}$ given in \eqref{eq:backwarddef}.
Here $u$ is either a theoretical quantile $F^\leftarrow(1-k/n)$ or the corresponding order statistic $X_{n-k:n}$. We focus on the differences between the estimators for the different types of thresholds. A detailed comparison of the finite-sample performance of the forward and the backward estimator with random threshold is given in \cite{Davis.2018}. To keep the presentation short, here we report detailed results only for lag $t=1$ and $x\in\{1/2,1\}$. Further results are given in the supplementary material.

We consider two classes of time series models: GARCH(1,1) time series and Markovian time series.
Recall that the GARCH(1,1) model is given by
\begin{equation*}
X_t=\sigma_t\varepsilon_t \quad\text{ with }\quad \sigma_t^2=\alpha_0+\alpha_1X_{t-1}^2+\beta_1\sigma_{t-1}^2,
\end{equation*}
$t\in\Z$,  where $\alpha_0\in(0,\infty),$ $\alpha_1,\beta_1\in[0,\infty)$, and $\varepsilon_t$, $t\in\Z$, are iid innovations. Under suitable conditions, the GARCH(1,1) time series are regularly varying \citep[Theorem 3.1]{Basrak.2002b} and $\beta$-mixing \citep[Theorem 8]{Andersen.2009}.
The innovations  are chosen either standard normal (leading to the so-called nGARCH model) or  standardized $t$-distributed with 4 degrees of freedom (denoted by tGARCH). Moreover, we use the parameters $\alpha_0=0.1$, $\alpha_1=0.14$ and $\beta_1=0.84$. This choice results in a tail index $\alpha$ of about 4.02 in the nGARCH case and 2.6 for the tGARCH model.

The distribution of the forward spectral tail process $(\Theta_t)_{t\in\N}$ is given in Proposition 6.2 of   \cite{Ehlert.2015}. In particular the marginals can be represented as follows:
\begin{equation}\label{s1}
\La(\Theta_t)=\La\bigg(\frac{\tilde{\varepsilon}_t}{|\tilde{\varepsilon}_0|}\prod_{i=1}^t(\alpha_1\tilde{\varepsilon}_{t-i}^2+\beta_1)^{1/2}\bigg)
\end{equation}
for all $t\in\N$, where $\tilde{\varepsilon}_t$, $t\in\N_0$, are independent and, except for $\tilde{\varepsilon}_0$, distributed as $\varepsilon_1$. The random variable $\tilde{\varepsilon}_0$ has density $h(x)=g(x){|x|^{\alpha}}/{E[|\varepsilon_0|^{\alpha}]},$ $x\in\R$, with $g$ denoting the density of $\varepsilon_1$ and $\alpha$ the index of regular variation of $(X_t)_{t\in\Z}$.

The distribution of a stationary Markovian time series is determined by its one-dimensional marginal distribution and the copula of two consecutive observations $(X_0,X_1)$. Such a time series is regularly varying if the one-dimensional marginal distribution has a density and is regularly varying with balanced tails, and the copula satisfies certain regularity conditions locally at the corners of its domain $[0,1]^2$. The forward spectral process $(\Theta_t)_{t\ge 0}$ is then a geometric random walk which is determined by the copula of $(X_0,X_1)$, the tail index $\alpha$ of the marginal distribution and the relative weight of the left and right tails. See \cite[Proposition 5.1]{Drees.2015} for details.
\newpage
We simulated Markovian time series with Student's $t$-distribution $F_{t,\nu}$ with $\nu=4$ degrees of freedom as marginal and either a copula of a bivariate $t$-distribution
\begin{equation*}
C_{\nu,\rho}^t(u,v)=\int_{-\infty}^{F_{t,\nu}^{\leftarrow}(u)}\int_{-\infty}^{F_{t,\nu}^{\leftarrow}(v)}\frac{1}{2\pi(1-\rho^2)^{1/2}}\left(1+\frac{x^2-2\rho xy+y^2}{\nu(1-\rho^2)}\right)^{-(\nu+2)/2}\ \textrm{d}x\ \textrm{d}y,
\end{equation*}
with $\nu=4$ degrees of freedom and  $\rho\in\{0.25, 0.5, 0.75\}$,
or a Gumbel-Hougaard copula
\begin{align*}
C_{\vartheta}^{gum}(u,v)=\exp\left(-[(-\log u)^{\vartheta}+(-\log v)^{\vartheta}]^{1/\vartheta}\right)
\end{align*}
with $\vartheta\in\{1.2, 1.5, 2.0\}$. For brevity's sake, here we report detailed results only for the former model with $\rho=0.25$, referred to as tCopula model in what follows. Its index of regular variation is $\alpha=4$; the distribution of $\Theta_1$ is given in \cite[Example A.1]{Drees.2015}.

To investigate the performance of the forward and backward estimator, we generate $1\,000$ GARCH(1,1) time series of length $n=2\,000$ for each of the models. The forward and backward estimators are based on theoretical quantiles (TQ) of $|X_0|$ at level $\beta\in\{0.9,0.95\}$ resp.\ on the corresponding order statistics (OS).
Since no closed form of the marginal distribution is available for the GARCH models, we calculate
the quantile of $|X_0|$ via Monte Carlo simulations.  More precisely, it is approximated by the average of the $\lfloor m\beta\rfloor$-th order statistics of $|X_1|,\dots,|X_m|$ of a time series of length $m=10^8$, obtained in 100  Monte Carlo simulations. The empirical standard deviation of these 100 realisations divided by $\sqrt{100}$ gives an estimate of the standard deviation of this approximation of the true quantile  (see Table \ref{tab1}).

The true probabilities $P\{\Theta_1>x \}$ (reported in Table \ref{tab2}) were calculated numerically. In addition, in Table \ref{tab3} we give the so-called pre-asymptotic quantities
\begin{align*}
p_{\beta}(x)&:=P\left(\frac{X_1}{|X_0|}>x\ \bigg|\ |X_0|>F^{\leftarrow}(\beta)\right)\quad\text{ and }\\ e_{\beta}(x)&:=E\bigg[\left|\frac{X_{-1}}{X_0}\right|^{1/a_\beta}\1\{X_0/|X_{-1}|>x\}\ \bigg|\ |X_0|>F^{\leftarrow}(\beta)\bigg] \quad \text{ with}\\
a_\beta & := E\Big[ \log\big(|X_0|/F^{\leftarrow}(\beta)\big) \ \Big|\ |X_0|>F^{\leftarrow}(\beta)\Big]
\end{align*}
for $x\in\{1/2, 1\}$ and $\beta\in\{0.9,0.95\}$. Note that the Hill-type estimator based on exceedances over a true quantile is an empirical counterpart to $a_\beta$, and
the forward and the backward estimators can be considered empirical counterparts to $p_{\beta}(x)$ and $e_{\beta}(x)$, respectively. One might thus expect that the estimates are concentrated around the corresponding pre-asymptotic values rather than the limit probability $P\{\Theta_1>x \}$.
These values can be calculated numerically for the Markovian copula models, whereas for the GARCH(1,1) model we again resorted to Monte Carlo simulations with the same design  as used for the approximation of the true quantiles (but with times series of length $m=10^7$).

Each of the following figures, which visualize our main findings, comprise four plots.  The two upper plots correspond to the forward estimator, the lower plots show results for the backward estimator. On the left hand side the results for the estimators based on exceeances over the theoretical resp.\ empirical $90\%$ quantile are displayed, while the plots on the right hand side correspond to the $95\%$ level.

Figure \ref{fig1} shows Q-Q plots of the estimators of $P\{\Theta_1>1\}\approx 0.0549$ based on theoretical quantiles vs the one based on order statistics in the nGARCH model.
Overall, the Q-Q plots are close to the main diagonal, which confirms our asymptotic results. However,
for small values the points usually lie below the main diagonal. This means that the estimators based on exceedances over theoretical quantiles tend to underestimate the true value more severely. A closer inspection of the simulation results reveals that this effect mainly occurs when few absolute observations exceed the quantile, which leads to an unreliable estimate of the TQ-version (whereas the number of exceedances is fixed in the OS-version). Indeed, in particular for $\beta=0.95$, in some simulations the TQ-versions of the estimators are very close to 0.

\begin{figure}[tb]
	\centering
	\includegraphics[width=1\textwidth]{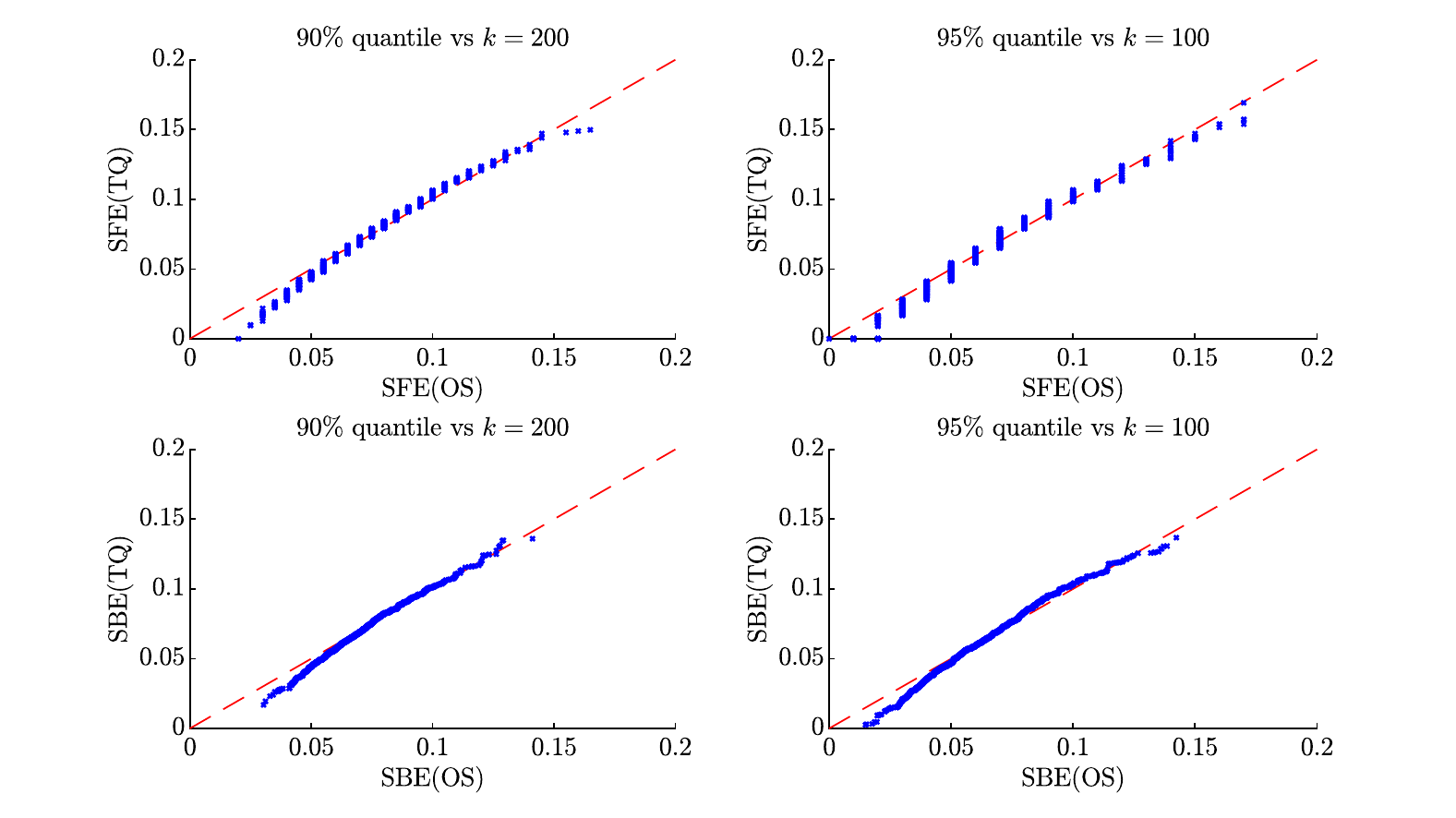}
	\caption{ Q-Q plots of the forward (top) and backward (bottom) estimator of $P\{\Theta_1>1\}$ using OS as threshold compared with TQ as threshold in the nGARCH model; left: $90\%$ level; right: $95\%$ level. The dashed red line is the main diagonal.}
	\label{fig1}
\end{figure}

Note that the discrete nature of the $x$-coordinates of the points in the upper plots is due to the fact that the forward estimator based on the exceedances over $X_{n-k:n}$ can assume only values in $\{i/k \mid 0\leqslant  i\leqslant k\}$. This is also obvious from the corresponding empirical cdfs shown in Figure \ref{fig2}. The solid black vertical line indicates the probability $P\{\Theta_1>1\}$, whereas the dash-dotted black line shows the pre-asymptotic values  $p_\beta(1)$ (for the forward estimator) resp.\ $e_\beta(1)$ (for the backward estimator). As expected, the cdfs are approximately centered at the corresponding pre-asymptotic values.

These graphs again show that the estimators which are based on theoretical quantiles tend to underestimate the true value more often. However, due to the difference between the pre-asymptotic value and their limit, all estimators usually overestimate 

\begin{figure}[H]
	\centering
	\includegraphics[width=1\textwidth]{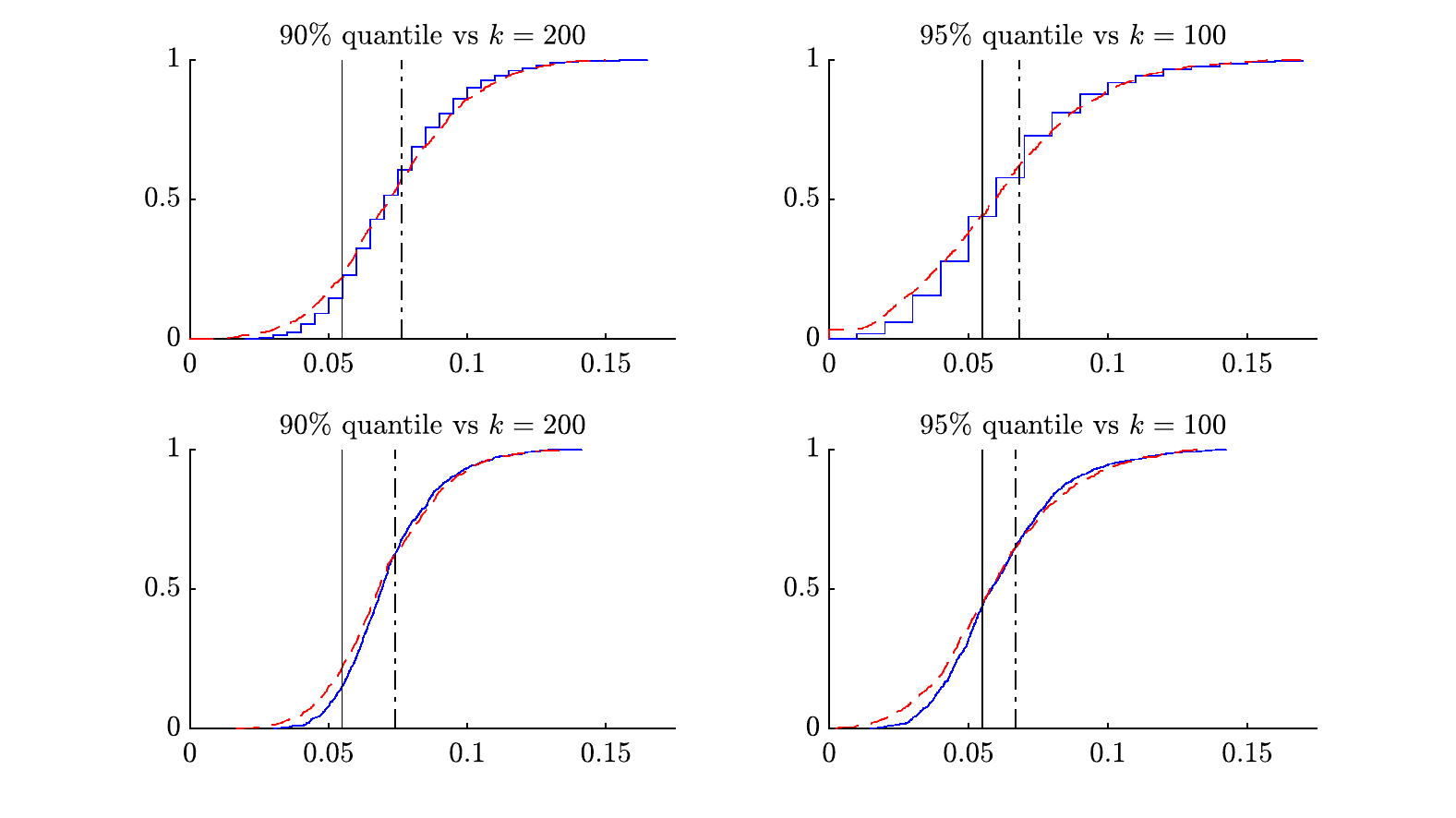}
	\caption{ Empirical cdfs of the forward (top) and backward (bottom) estimators of $P\{\Theta_1>1\}$ using TQ (dashed red line) and OS (solid blue line) as thresholds in the nGARCH model. The solid black vertical line indicates the probability $P\{\Theta_1>1\}$, whereas the dash-dot black vertical line indicates the pre-asymptotic values $p_{\beta}(1)$ in the upper two plots and $e_{\beta}(1)$ in the lower two plots.}
	\label{fig2}
\end{figure}
\begin{figure}[H]
	\centering
	\includegraphics[width=1\textwidth]{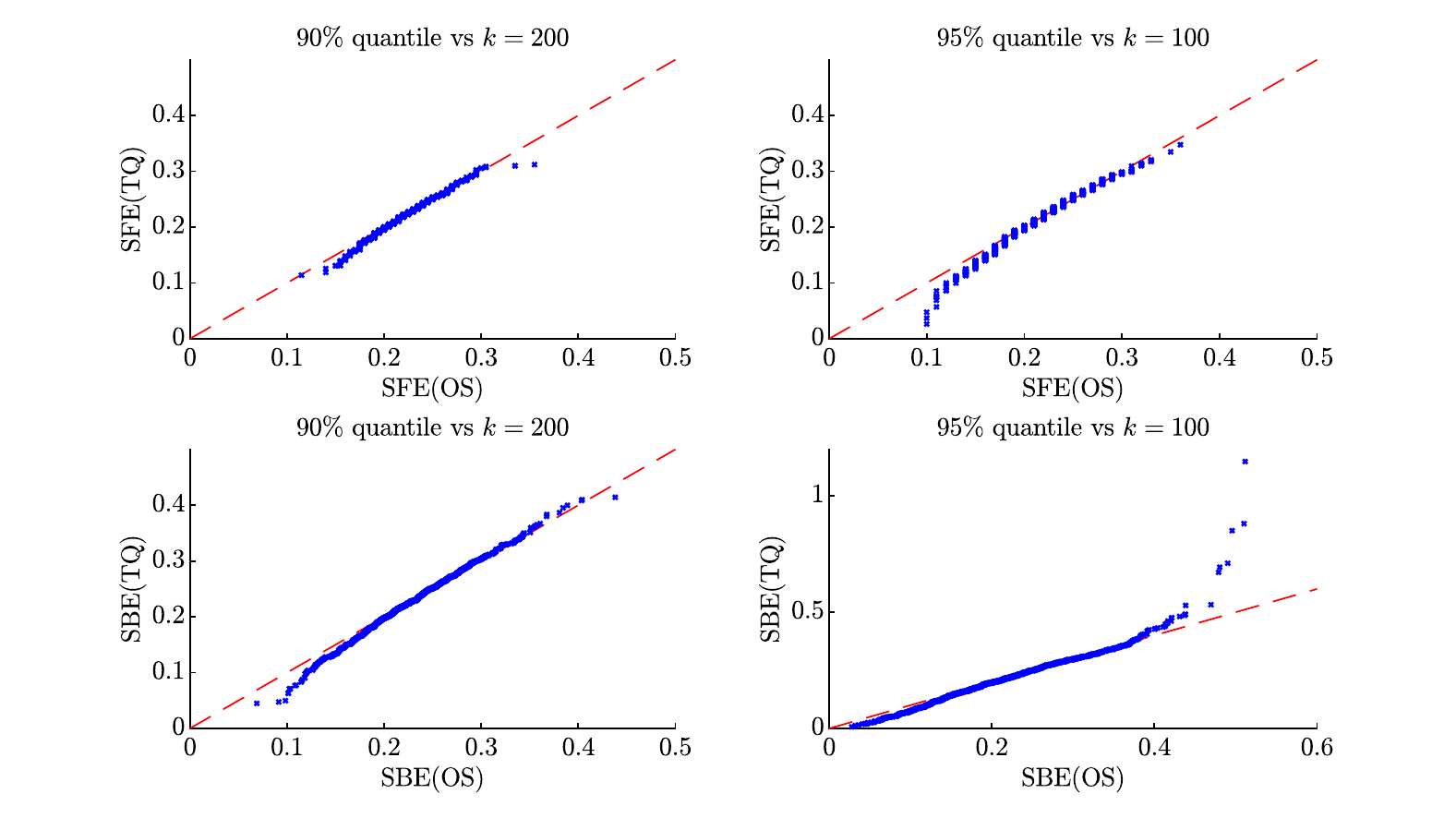}
	\caption{ Q-Q plots of the forward (top) and backward (bottom) estimator of $P\{\Theta_1>1/2\}$ using OS as threshold compared with TQ as threshold in the nGARCH model; left: $90\%$ level; right: $95\%$ level. The dashed red line is the main diagonal.}
	\label{fig3}
\end{figure}
\noindent the asymptotic value. Furthermore, as already reported in \cite{Davis.2018}, the backward estimator of $P\{\Theta_t>1\}$ performs better than the forward estimator.

Figure \ref{fig3} displays  Q-Q plots for the estimators of $P\{\Theta_1 >1/2\}$. By and large, the same effects occur as in Figure \ref{fig1}. In addition, in some cases the TQ-version of the backward estimator strongly overestimates the true value; in one simulation with $\beta=0.95$, it even gives an estimate larger than 1. This  is again mainly due to simulations when relatively few observations exceed the quantile in absolute value. Indeed, for $x=1/2$, the forward estimator is preferable in the nGARCH model.

The results for the tGARCH model and for lags $t\in\{3,5\}$ are qualitatively the same. For this reason, they are not shown here.

\begin{figure}[bt]
	\centering
	\includegraphics[width=1\textwidth]{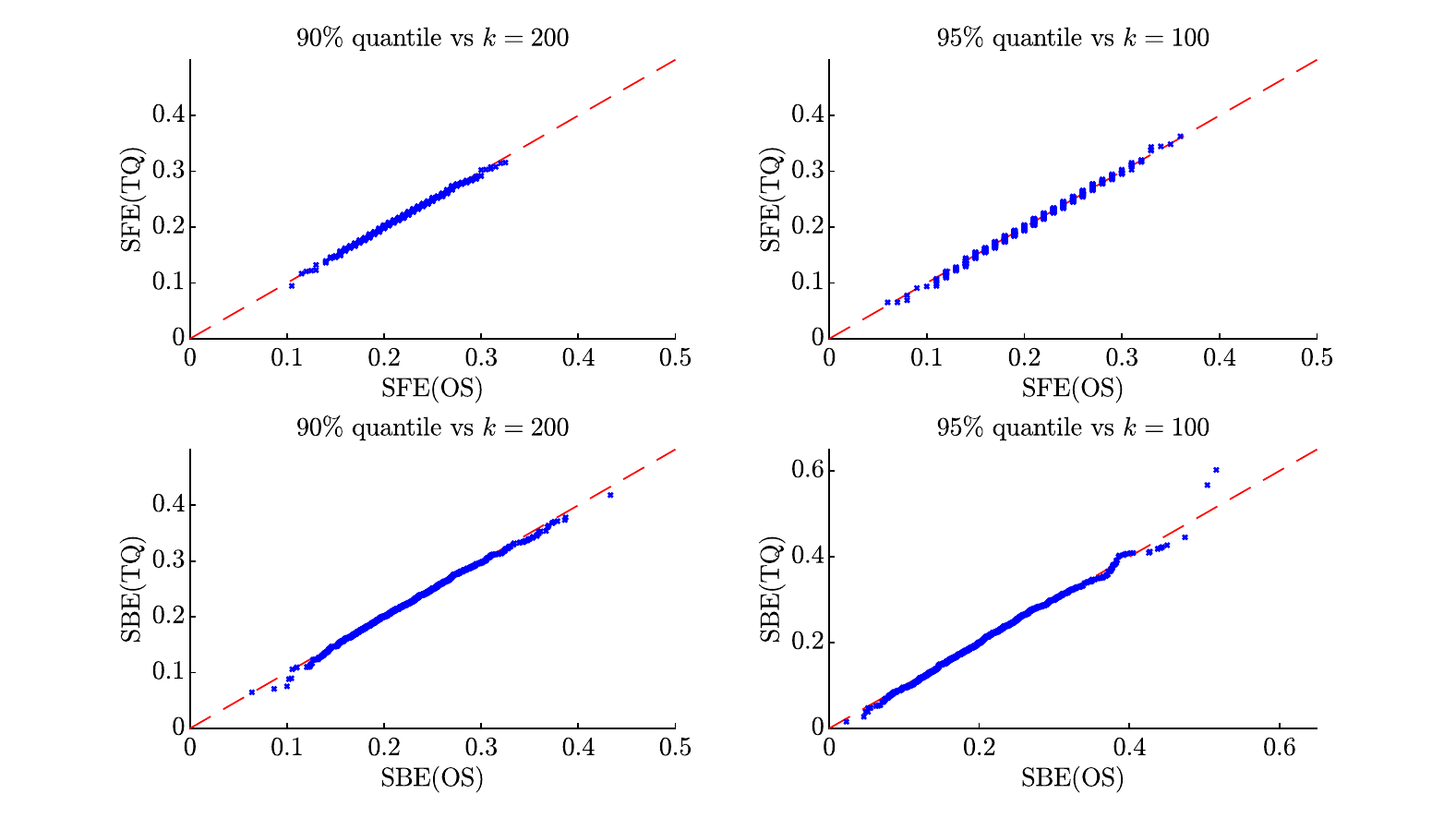}
	\caption{ Q-Q plots of the forward (top) and backward (bottom) estimator of $P\{\Theta_1>1/2\}$ in the tCopula model with $\rho=0.25$ using OS as threshold compared with TQ as threshold; left: $90\%$ level; right: $95\%$ level. The dashed red line is the main diagonal.}
	\label{fig4}
\end{figure}

Next we investigate the estimators of $P\{\Theta_1>1/2\}\approx0.1831$ in the Markovian tCopula model  with $\rho=0.25$. The Q-Q plots of the TQ-versions vs the OS-versions (Figure \ref{fig4}) are even closer to the main diagonal than in the GARCH(1,1) models, with only minor fluctuations around the main diagonal for large values of the backward estimator. The results  for estimators of $P\{\Theta_1>1/2\}$ in the other Markovian copula models and for larger lags look similar.

So far we have compared the {\em distributions} of the two versions of forward resp.\ backward estimators, which according to Theorem \ref{theo:asnormest} have the same asymptotic behavior. Except for some differences in the tails, the distributions were close together also for finite sample sizes. Therefore, it suggests itself to examine whether there is a closer relationship between both version, in that for each simulation both versions give similar {\em estimates}, or that the difference between both versions is typically of smaller order than the variability of each of the estimators.

In Figure \ref{fig5} we plot the estimates of $P\{\Theta_1>1/2\}$ based on exceedances over TQ vs the estimates based on exceedances over OS in the usual format for the tCopula model with $\rho=0.25$.
Although the relationship between the realizations of both versions is obviously weaker than the relation between the order statistics (shown in Figure \ref{fig4}), the estimator based on exceedances over a random threshold seems quite an accurate predictor for the estimator based on exceedances over a theoretical quantile. Indeed, the variance of the difference of both versions is merely between 6.5\% (for the forward estimator to the 90\% level) and 11.4\% (for the backward estimator to the 95\% level) of the variance of the TQ-version.

\begin{figure}[bt]
	\centering
	\includegraphics[width=1\textwidth]{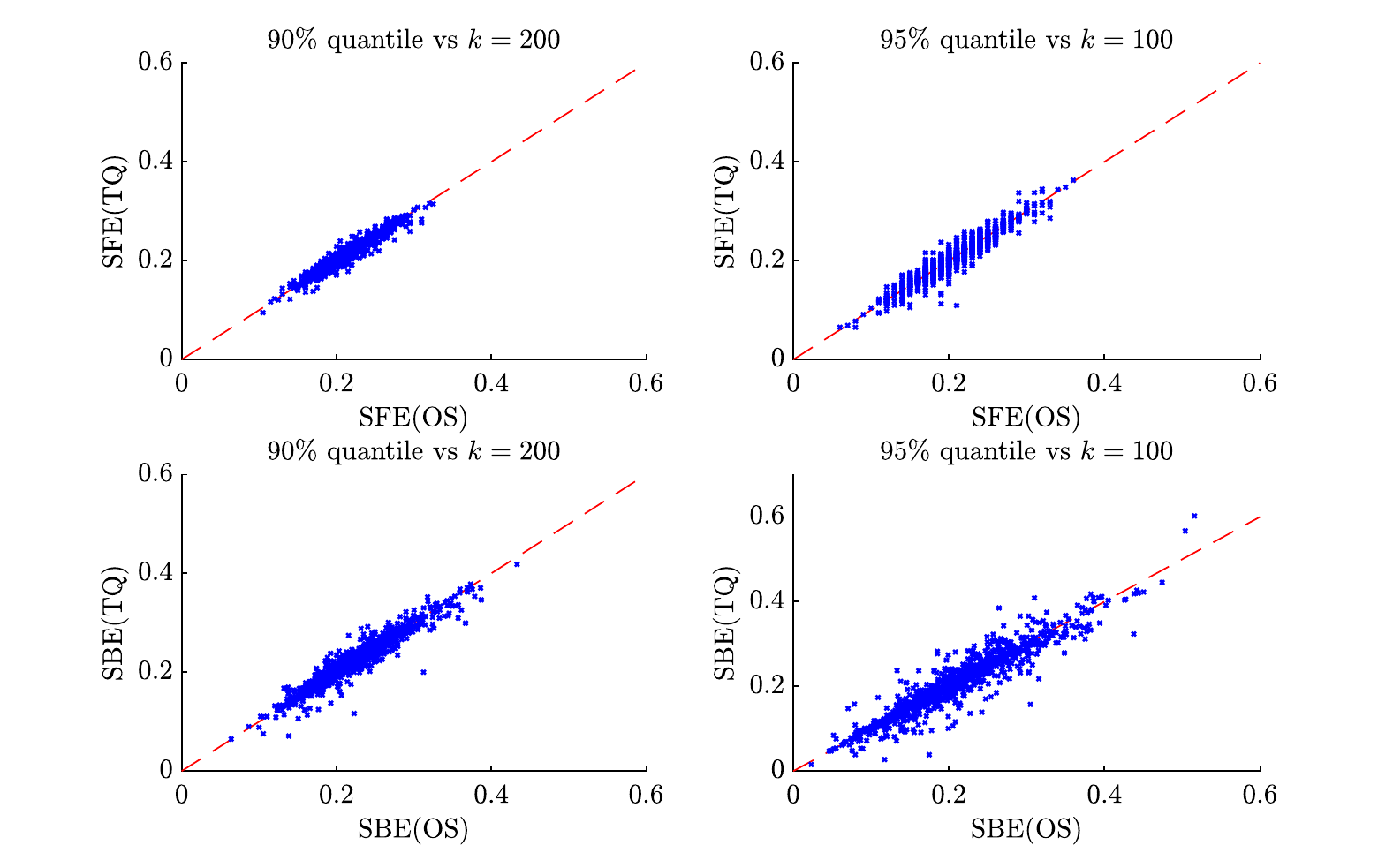}
	\caption{Forward (top) resp.\ backward (bottom) estimator  of $P\{\Theta_1>1/2\}$ in tCopula model with $\rho=0.25$ using TQ as threshold vs estimator using OS as threshold; left: $90\%$ level; right: $95\%$ level. The dashed red line is the main diagonal.}
	\label{fig5}
\end{figure}

\newpage
\section*{Appendix A: Tables}
\setcounter{section}{1}
\renewcommand{\thesection}{\Alph{section}}
\label{app:tables}

\begin{table}[H]
	\small
	\begin{center}
		\begin{tabular}{l | c c}
			& \multicolumn{2}{c}{$\beta$}\\
			model & 0.9 & 0.95 \\ \hline \rule{0pt}{2.2ex}
			nGARCH & 3.3931 ($1.6\times 10^{-4}$) & 4.3695 ($2.7\times 10^{-4}$)\\ \rule{0pt}{2.2ex}
			tGARCH & 2.6349 ($1.5\times 10^{-4}$) & 3.7005 ($2.7\times 10^{-4}$)\\ \rule{0pt}{2.2ex}
			Copula & 2.1318 & 2.7764
		\end{tabular}
		\caption{Approximate theoretical quantiles $F^{\leftarrow}(\beta)$, for $\beta=90\%$ or $95\%$ (with estimated standard deviations in parentheses). In the copula model $F^{\leftarrow}(\beta)$ is completely determined by the marginal $t_4$ distribution.}
		\label{tab1}
	\end{center}
\end{table}

\begin{table}[H]
	\small
	\centering
	\begin{tabular}{l l | c | c}
		model & & $P\{\Theta_1>1\}$ & $P\{\Theta_1>1/2\}$ \\ \hline \rule{0pt}{2.2ex}
		nGARCH & & 0.0549 & 0.2022\\ \hline \rule{0pt}{2.2ex}
		tGARCH & & 0.0450 & 0.1415\\ \hline \rule{0pt}{2.2ex}
		tCopula with & $\rho=0.25$ & 0.0445 & 0.1831\\
		& $\rho=0.50$ & 0.0662 & 0.2623\\
		& $\rho=0.75$ & 0.1096 & 0.3929\\ \hline \rule{0pt}{2.2ex}
		gumCopula with & $\theta=1.2$ & 0.0546 & 0.2145\\
		& $\theta=1.5$ & 0.1031 & 0.3756\\
		& $\theta=2.0$ & 0.1464& 0.4688
	\end{tabular}
	\caption{Probabilities $P\{\Theta_1>1\}$ and $P\{\Theta_1>1/2\}$ in each model.}
	\label{tab2}
\end{table}

\begin{table}[H]
	\scriptsize
	\begin{center}
		\begin{tabular}{l c | c c | c c}
			model & & \multicolumn{2}{c}{nGARCH} & \multicolumn{2}{|c}{tGARCH}\\
			& $\beta$ & 0.9 & 0.95 & 0.9 & 0.95\\ \hline \rule{0pt}{2.2ex}
			$p_{\beta}(1)$ & SFE with OS & 0.0763 ($3\times 10^{-5}$) & 0.0683 ($5\times 10^{-5}$) & 0.0663 ($3\times 10^{-5}$) & 0.0575 ($4\times 10^{-5}$)\\ \hline \rule{0pt}{2.2ex}
			$e_{\beta}(1)$ & SBE with OS & 0.0740 ($3\times 10^{-5}$) & 0.0669 ($5\times 10^{-5}$) & 0.0704 ($3\times 10^{-5}$) & 0.0610 ($4\times 10^{-5}$)\\ \hline \rule{0pt}{2.2ex}
			$p_{\beta}(1/2)$ & SFE with OS & 0.2283 ($4\times 10^{-5}$) & 0.2189 ($6\times 10^{-5}$) & 0.1820 ($4\times 10^{-5}$) & 0.1668 ($6\times 10^{-5}$)\\ \hline \rule{0pt}{2.2ex}
			$e_{\beta}(1/2)$ & SBE with OS & 0.2300 ($7\times 10^{-5}$) & 0.2188 ($10^{-4}$) & 0.1842 ($5\times 10^{-5}$) & 0.1681 ($9\times 10^{-5}$)\\
		\end{tabular}
		\caption{Estimated $p_{\beta}(x)$ and $e_{\beta}(x)$ for $x\in\{1,1/2\}$ (with estimated standard deviations in parentheses).}
		\label{tab3}
	\end{center}
\end{table}

\section*{Appendix B: Stochastic Recurrence Equations}
\setcounter{equation}{0}
\setcounter{section}{2}
\renewcommand{\thesection}{\Alph{section}}
\label{app:ex}

Consider the stochastic recurrence equation
\begin{equation}\label{eq:SRE}
X_t=C_tX_{t-1}+D_t,\quad t\in\Z,
\end{equation}
where $(C_t,D_t)$, $t\in\Z$, is a sequence of iid $[0,\infty)^2$-valued random variables. It is well known  that there exists a unique strictly stationary causal solution, provided $E[\log C_1] <0$ and $E[\log^+ D_1] <\infty$ \citep[Cor.~2.2]{Basrak.2002b}. In addition, assume that the distribution of $C_1$ is not concentrated on a lattice and that there exists $\alpha>0$ such that $E[C_1^\alpha] =1$, $E[ C_1^\alpha \log^+ C_1 ] < \infty$ and $E[ D_1 ^\alpha] < \infty$. Then the time series is regularly varying  with index $\alpha$ \citep[Rem.~2.5, Cor.~2.6]{Basrak.2002b}.

\cite{Drees.2015} have shown that condition (B) holds for suitably chosen (logarithmically increasing) $r_n$, provided $(\log n)^2/n=o(v_n)$ and $v_n=o(1/\log n)$, and that a milder version of condition (C) is satisfied. Here, we will show that our strengthened  condition (C) is fulfilled, too, if we assume in addition that $r_n^{1+2\delta}v_n=O(1)$.

Let $\Pi_{i,j}:=\prod_{l=i}^{j}C_l$ and $V_{i,j}:=\sum_{l=i}^j\Pi_{l+1,j}D_l$. Iterating \eqref{eq:SRE} yields $X_k=\prod_{j+1,k}X_j+V_{j+1,k}$. Define $u_{n,\eps}=(1-\eps)u_n$ and $v_{n,\eps}:=P\{X_0>u_{n,\eps}\}$.

Consider $g:\R\to\R$ with $0\leqslant g(x)\leqslant ax^{\tau}+b$ for some $a,b\geqslant 0$, $\tau\in(0,\alpha)$ and all $x\in\R$. Using the Potter bounds, one can show that the random variables $g(X_0/u)\1\{X_0>u\}/P\{X_0>u\}$, $u>u_0$, are uniformly integrable. Hence, for sufficiently large $u$,
\begin{equation}\label{eq:ineq0}
E[g(X_0/u)\1\{X_0>u\}]\leqslant 2E[g(Y_0)]P\{X_0>u\}.
\end{equation}

Under the above conditions, one has $\rho:=E[C_1^{\xi}]<1$ for any  $\xi\in(0,\alpha)$.
Thus, by the generalized Markov inequality and the independence of the random variables $C_l$,
\begin{equation} \label{eq:ineq1}
P\{\Pi_{j+1,k}>u_{n,\eps}/(2t)\}  \leqslant \rho^{k-j}(2t/u_{n,\eps})^{\xi}.
\end{equation}
Inequality \eqref{eq:ineq0}, $V_{1,k}\leqslant X_k$ and the Potter bounds imply
\begin{align}
E[(X_0/u_{n,\eps})^{\xi}\1\{X_0>u_{n,\eps}/2\}] & \leqslant 2^{1+\alpha}v_{n,\eps}E[Y_0^{\xi}], \label{eq:ineq2}\\
P\{V_{1,k}>u_{n,\eps}/2\} & \leqslant P\{X_k>u_{n,\eps}/2\}\leqslant 2^{1+\alpha}v_{n,\eps} \label{eq:ineq3}
\end{align}
for all $k\in\N$ and sufficiently large $n$.
Moreover, it was shown in  \citep[Example A.3]{Drees.2015} that
there exists a constant $c>0$ such that
\begin{align}
P\{\min\{X_0,X_k\}>u\}& \leqslant cP\{X_0>u\}(P\{X_0>u\}+\rho^k), \label{eq:ineq4}\\
P\{\Pi_{1,k}X_0>u/2,X_0>u\} & \leqslant 2^{\xi+1} \rho^k E[Y_0^\xi] P\{X_0>u\} \label{eq:ineq5}
\end{align}
for all $k\in\N$ and all $u$ sufficiently large.

By independence of $(V_{j+1,k},\Pi_{j+1,k})$ and $(X_0,X_j)$, one has
\begin{align*}
P\{\min\{X_0,X_j,X_k\}>u_{n,\eps}\}&\leqslant P\{\min\{X_0,X_j\}>u_{n,\eps},V_{j+1,k}>u_{n,\eps}/2\}\\
&\qquad+P\{\min\{X_0,X_j\}>u_{n,\eps},\Pi_{j+1,k}X_j>u_{n,\eps}/2\}\\
&=P\{\min\{X_0,X_j\}>u_{n,\eps}\}P\{V_{j+1,k}>u_{n,\eps}/2\}\\
&\qquad+\int_{(u_{n,\eps},\infty)^2}P\{\Pi_{j+1,k}>u_{n,\eps}/(2t)\}\ P^{(X_0,X_j)}(\mathrm{d}(s,t))\\
&\leqslant c2^{1+\alpha}v_{n,\eps}^2(v_{n,\eps}+\rho^j)\\
&\qquad+\rho^{k-j}2^{\xi}E[(X_j/u_{n,\eps})^{\xi}\1\{\min\{X_0,X_j\}>u_{n,\eps}\}]
\end{align*}
where in the last step we have used \eqref{eq:ineq1}, \eqref{eq:ineq3} and \eqref{eq:ineq4}.
Using $(a+b)^\xi\le 2^\xi(a^\xi+b^\xi)$ for all $a,b>0$, \eqref{eq:ineq2}, \eqref{eq:ineq5}, $V_{1,j}\leqslant X_j$ and the independence of $X_0$ and $V_{1,j}$, we can bound the last expected value as follows:
\begin{align*}
E[&(X_j/u_{n,\eps})^{\xi}\1\{\min\{X_0,X_j\}>u_{n,\eps}\}]\\
& \leqslant 2^{\xi}E\Big[\big( (\Pi_{1,j}X_0/u_{n,\eps})^{\xi}+(V_{1,j}/u_{n,\eps})^{\xi}\big)\1\{X_0>u_{n,\eps},\Pi_{1,j}X_0+V_{1,j}> u_{n,\eps}\}\Big]\\
& \leqslant 2^{\xi}E\Big[(\Pi_{1,j}X_0/u_{n,\eps})^{\xi}\1\{X_0>u_{n,\eps}\} + (V_{1,j}/u_{n,\eps})^{\xi}\1\{X_0>u_{n,\eps},V_{1,j}>u_{n,\eps}/2\} \\
& \hspace{3cm}+(V_{1,j}/u_{n,\eps})^{\xi}\1\{X_0>u_{n,\eps},V_{1,j}\leqslant u_{n,\eps}/2,\Pi_{1,j}X_0>u_{n,\eps}/2\}\big)\Big]\\
&\leqslant 2^{\xi}\Big[\rho^j  E\big[(X_0/u_{n,\eps})^{\xi}\1\{X_0>u_{n,\eps}/2\}\big] + v_{n,\eps}E\big[(X_j/u_{n,\eps})^{\xi}\1\{X_j>u_{n,\eps}/2\}\big]\\
& \hspace{3cm}  + 2^{-\xi} P\{X_0>u_{n,\eps},\Pi_{1,j}X_0>u_{n,\eps}/2\}      \Big]\\
&\leqslant 2^{\xi}\Big[(\rho^j+v_{n,\eps}) 2^{1+\alpha} v_{n,\eps}E[Y_0^{\xi}]+2\rho^j v_{n,\eps}E[Y_0^{\xi}]\Big].
\end{align*}
To sum up, we have shown that
\begin{align*}
P\{&\min\{X_0,X_j,X_k\}>u_{n,\eps}\}\\
& \leqslant c2^{1+\alpha}v_{n,\eps}^2(v_{n,\eps}+\rho^j)+\rho^{k-j}2^{2\xi+1}E[Y_0^{\xi}]v_{n,\eps} ((2^\alpha+1)\rho^j+2^{\alpha}v_{n,\eps}).
\end{align*}
This yields
$$
P(\min\{X_j,X_k\}>u_{n,\eps}\ |\ X_0>u_{n,\eps})\leqslant C(v_{n,\eps}^2+v_{n,\eps}(\rho^j+\rho^{k-j})+\rho^k)=:\tilde s_n(j,k)
$$
for a suitable constant $C>0$.
Now, note that $\tilde s_n(j,k)\to C\rho^{k}=:\tilde s_{\infty}(j,k)$ for all $j\leqslant k$, and
\begin{align*}
\sum_{1\leqslant j\leqslant k\leqslant r_n}\tilde s_n(j,k)&=C\Big(\frac{r_n(r_n+1)}{2}v_{n,\eps}^2+v_{n,\eps}\sum_{l=0}^{r_n-1}(r_n-l)\rho^l\\
&\qquad+v_{n,\eps}\sum_{j=1}^{r_n}(r_n-j+1)\rho^j+\sum_{k=1}^{r_n}k\rho^k\Big)\\
&\to C\sum_{k=1}^\infty k\rho^k=\sum_{1\leqslant j\leqslant k<\infty}\tilde s_{\infty}(j,k)<\infty,
\end{align*}
because $r_n v_{n,\eps}\to 0$.  Thus, condition \eqref{eq:snjkdef} is fulfilled.

Next, we verify condition \eqref{eq:snkdef} which is equivalent to
\begin{align}\label{eq:altsnjk}
\lim_{L\to\infty}\limsup_{n\to\infty}\sum_{j=L+1}^{r_n}E\big[\psi(u_{n,\eps}^{-1}X_0)\psi(u_{n,\eps}^{-1}X_k)\ \big|\ X_0>u_{n,\eps}\big]=0
\end{align}
for $\psi(x)=\max\{\log x,\1\{x>1\}\}$. This can be done by direct calculations, but here we give a more elegant proof using general results for Markov processes under the additional assumptions that the time series is aperiodic and irreducible. This is e.g.\ true if $(C_1,D_1)$ is absolutely continuous; see \cite{Buraczewskietal.2016}, Proposition 2.2.1 and Lemma 2.2.2.
According to Lemma 4.3 of \cite{Kulik.2018}, convergence \eqref{eq:altsnjk} holds when conditions (i)--(vi) in Assumption 2.1 of this paper are fulfilled. Since $(X_t)_{t\in\Z}$ is a regularly varying Markov chain, condition (i) and (ii) are trivial. The arguments given in subsection 5.2 of \cite{Mikosch.2013} show that the Lyapunov drift condition (iii) holds with $V(x)=1+|x|^p$ for any $p\in(0,\alpha)$.  The small set condition (iv) follows from subsection 2.2 of this paper in combination with Theorem 9.4.10 and Corollary 14.1.6 of \cite{Doucetal.2018}.
With the above choice of $V$, condition (v) is obvious. Using \eqref{eq:ineq0} and the Potter bounds one may conclude, for all $s>0$,
\begin{align*}
E[V(X_0)\1\{X_0>su_{n,\eps}\}]&= E\big[X_0^p\1\{X_0>su_{n,\eps}\}\big]+P\{X_0>su_{n,\eps}\}\\
&\leqslant \big(2(su_{n,\eps})^pE[Y_0^p]+1\big)P\{X_0>su_{n,\eps}\}\\
&\leqslant C\big((su_{n,\eps})^p+1\big)s^{-\alpha+\eta}v_{n,\eps}
\end{align*}
for some  $\eta$ with arbitrarily small modulus ($\eta$ is positive when $s>1$ and negative for $s\in(0,1)$) and sufficiently large $C>0$, $n\in\N$. Thus
\begin{align*}
\limsup_{n\to\infty}\frac{1}{u_{n,\eps}^{p}v_{n,\eps}}E[V(X_0)\1\{X_0>su_{n,\eps}\}] &\leqslant \limsup_{n\to\infty} C\big(s^p+u_{n,\eps}^{-p}\big)s^{-\alpha+\eta}\\
&=Cs^{-\alpha+p+\eta}<\infty
\end{align*}
such that condition (vi) is also satisfied. Hence, condition \eqref{eq:snkdef} is fulfilled.

It remains to prove \eqref{eq:psibdd}. Since to all $p>0$ there exists $c_p>0$ such that\linebreak
$(\log^+ x)^{1+\delta}\leqslant c_p x^p \1\{x>1\}$, it suffices to show that, for some $p,\tilde p>0$,
\begin{equation} \label{eq:powerbdd}
\sum_{k=1}^{r_n} \bigg( E\Big[ \Big(\frac{X_k}{u_{n,\eps}}\Big)^p \Big(\frac{X_0}{u_{n,\eps}}\Big)^{\tilde p}\1\{X_k>u_{n,\eps}\} \ \Big|\ X_0>u_{n,\eps}\Big]\bigg)^{1/(1+\delta)}
\end{equation}
is bounded. By induction, one can conclude from the drift condition that to all $p\in(0,\alpha)$ there exist $\beta\in(0,1)$ and $b>0$ such that $E[X_k^p\mid X_0=y]\leqslant \beta^k y^p+b/(1-\beta)$ \citep[Prop.\ 14.1.8]{Doucetal.2018}. Hence
\begin{align*}
E\Big[ & \Big(\frac{X_k}{u_{n,\eps}}\Big)^p \Big(\frac{X_0}{u_{n,\eps}}\Big)^{\tilde p} \1\{X_k>u_{n,\eps}\} \ \Big|\  X_0>u_{n,\eps}\Big]\\
& \leqslant  v_{n,\eps}^{-1}\int_{u_{n,\eps}}^\infty u_{n,\eps}^{-(p+\tilde p)}E[X_k^p\mid X_0=y] y^{\tilde p}\ P^{X_0}(\mathrm{d}y)\\
& \leqslant \beta^k E\Big[\Big(\frac{X_0}{u_{n,\eps}}\Big)^{p+\tilde p}\ \Big|\ X_0>u_{n,\eps}\Big] + \frac{b}{1-\beta}u_{n,\eps}^{-p}E\Big[\Big(\frac{X_0}{u_{n,\eps}}\Big)^{\tilde p}\ \Big|\  X_0>u_{n,\eps}\Big]\\
& \leqslant 2\beta^k E[Y_0^{p+\tilde p}] + \frac{2b}{1-\beta}u_{n,\eps}^{-p}E[Y_0^{\tilde p}]
\end{align*}
for sufficiently large $n$, by \eqref{eq:ineq0}, provided $p+\tilde p<\alpha$. Choose $p\in(\alpha(1+\delta)/(1+2\delta),\alpha)$ and $\tilde p>0$ sufficiently small such that $p+\tilde p<\alpha$. Then \eqref{eq:powerbdd} can be bounded by a multiple of $\sum_{k=1}^{r_n} \beta^{k/(1+\delta)}+r_n u_{n,\eps}^{-p/(1+\delta)}$. By the regular variation of $X_0$ with index $\alpha$ and the choice of $p$, one has $u_{n,\eps}^{-p/(1+\delta)}=o(v_{n,\eps}^{1/(1+2\delta)})$. Thus, \eqref{eq:powerbdd} is bounded if $r_n^{1+2\delta}v_{n,\eps}$ is bounded.

\section*{Appendix C: Proofs}
\setcounter{equation}{0}
\setcounter{section}{3}
\setcounter{Thm}{0}
\renewcommand{\thesection}{\Alph{section}}
\begin{Lem}\label{lemma:CimpliesD}
	If condition \eqref{eq:snjkdef} from (C) holds, then
	\begin{equation*}
	E\bigg[\bigg(\sum_{i=1}^{r_n}\1\{X_i>(1-\eps)u_n\}\bigg)^3\bigg]=O(r_nv_n).
	\end{equation*}
\end{Lem}
\begin{proof}
	Let $v_{n,\eps}=P\{X_0>(1-\eps)u_n\}$. By regular variation and stationarity of $(X_t)_{t\in\Z}$
	\begin{align*}
	&E\bigg[\bigg(\sum_{i=1}^{r_n}\1\{X_i>(1-\eps)u_n\}\bigg)^{3}\bigg]\\
	&\qquad=\sum_{i,j,k=1}^{r_n}P\{\min\{X_i,X_j,X_k\}>(1-\eps)u_n\}\\
	&\qquad \leqslant r_nv_{n,\eps}+6\sum_{k=1}^{r_n-1}\sum_{j=1}^{k}\Big(1-\frac{k}{r_n}\Big) P\{\min\{X_0,X_j,X_k\}>(1-\eps)u_n\}\\
	&\qquad\leqslant r_nv_{n,\eps}+6r_nv_{n,\eps}\sum_{1\leqslant j\leqslant k\leqslant r_n-1}\tilde{s}_n(j,k)\\
	&\qquad=O(r_nv_n).
	\end{align*}
\end{proof}
Taking up the notation of \cite{Drees.2010}, we consider the empirical process $\tilde{Z}_n$ defined by
\begin{equation*}
\tilde{Z}_n(\psi):=(nv_{n})^{-1/2}\sum_{i=1}^n(\psi(X_{n,i})-E[\psi(X_{n,i})]),
\end{equation*}
where $\psi$ is one of the functions $\phi_{0,s},\phi_{1,s},\phi_{2,x,s}^t$ or $\phi_{3,y,s}^t$ (defined in \eqref{eq:phi0def}--\eqref{eq:phi3def}). The asymptotic normality of the Hill estimator and our main Theorem \ref{theo:asnormest} can be derived from the following result about the process convergence of $\tilde Z_n$.
\begin{Prop}\label{prop1}
	Let $(X_t)_{t\in\Z}$ be a stationary, regularly varying process with index $\alpha>0$. Suppose that the conditions (A($x_0$)), (B) and (C) are fulfilled for some $x_0\geqslant 0$. Then, for all $y_0\in[x_0,\infty)\cap(0,\infty)$, the sequence of processes $(\tilde{Z}_n(\psi))_{\psi\in\Phi}$ with index set
	\begin{equation*}
	\Phi:=\left\{\phi_{0,s},\phi_{1,s},\phi_{2,x,s}^t,\phi_{3,y,s}^t\ |\ s\in[1-\varepsilon,1+\varepsilon],x\geqslant x_0,y\geqslant y_0, |t|\in\{1,\dots,\tilde{t}\}\right\}
	\end{equation*}
	converges weakly in $l^{\infty}(\Phi)$ to a centered Gaussian process $Z$ with covariance function given by
	\begin{align} \label{eq:cov_emp_pr}
	\cov(Z(\psi_1),Z(\psi_2))&= E[\psi_1(\bar{Y}_0)\psi_2(\bar{Y}_0)]+ \sum_{k=1}^{\infty}\left(E[\psi_1(\bar{Y}_0)\psi_2(\bar{Y}_k)]+ E[\psi_1(\bar{Y}_k)\psi_2(\bar{Y}_0)]\right)\\
	&= \sum_{k=-\infty}^{\infty}E[\psi_1(\bar{Y}_0)\psi_2(\bar{Y}_k)] \nonumber
	\end{align}
	for $\psi_1,\psi_2\in\Phi$, where $\bar{Y}_k:=(Y_{k-\tilde{t}},\dots,Y_{k+\tilde{t}})\1\{Y_k>1\}.$
\end{Prop}
\begin{proof}
	Weak convergence of the finite-dimensional distributions of $(\tilde{Z}_n(\psi))_{\psi\in\Phi}$ can be established as in the proof of Proposition B.1 of \cite{Drees.2015}. Note that here the threshold $(1-\eps)u_n$ is used instead of $u_n$, while the components of $X_{n,i}$ are standardized with $u_n$. Moreover, we standardize the process using $v_n=P\{X_0>u_n\}$ instead of $P\{X_0>(1-\eps)u_n\}= (1-\eps)^{-\alpha}v_n(1+o(1))$. Therefore, we obtain as limiting covariance function
	\begin{align*} \cov&(Z(\psi_1),Z(\psi_2))= (1-\eps)^{-\alpha}\Big(E[\psi_1((1-\eps)\bar{Y}_0)\psi_2((1-\eps)\bar{Y}_0)] \\ & +\sum_{k=1}^{\infty}\left(E[\psi_1((1-\eps)\bar{Y}_0)\psi_2((1-\eps)\bar{Y}_k)] +E[\psi_1((1-\eps)\bar{Y}_k)\psi_2((1-\eps)\bar{Y}_0)]\right)\Big).
	\end{align*}
	Now recall that $ Y_k=Y_0 \Theta_k$ for a Pareto random variable $Y_0$ independent of the spectral process. Since $P\{(1-\eps)Y_0>1\}=(1-\eps)^\alpha$, $Y_0$ has the same distribution as $(1-\eps)Y_0$ conditionally on $\{(1-\eps)Y_0>1\}$, and $\psi_i(y_{-\tilde t},\ldots, y_{\tilde t})$ vanishes if $y_0\leqslant 1$, one has
	\begin{align*}
	E\big[\psi_1((1-\eps)\bar Y_0)\psi_2((1-\eps)\bar Y_k)\big]& = E\big[\psi_1((1-\eps)Y_0(\Theta_t)_{|t|\leqslant \tilde t})\psi_2((1-\eps)Y_0(\Theta_{k+t})_{|t|\leqslant \tilde t})\big]\\
	& = (1-\eps)^\alpha E\big[\psi_1(Y_0(\Theta_t)_{|t|\leqslant \tilde t})\psi_2(Y_0(\Theta_{k+t})_{|t|\leqslant \tilde t})\big]\\
	&= (1-\eps)^\alpha E\big[\psi_1(\bar Y_0)\psi_2(\bar Y_k)\big].
	\end{align*}
	Now, the asserted representation \eqref{eq:cov_emp_pr} is obvious. The second representation can be similarly concluded from the equation
	$$ E\Big[ \sum_{i=1}^{r_n} \psi_1(X_{n,i})\sum_{j=1}^{r_n} \psi_2(X_{n,j})\Big] = r_n\sum_{k=1-r_n}^{r_n-1} \Big( 1-\frac{|k|}{r_n}\Big) E\big[\psi_1(X_{n,0})\psi_2(X_{n,k})\big].
	$$

	To prove asymptotic equicontinuity of the processes (and thus their weak convergence), we apply Theorem 2.10 of \cite{Drees.2010}. To this end, we must verify the conditions (D1), (D2'), (D3), (D5) and (D6) of this paper. Except for condition (D6), all assumptions of the theorem can be established by similar arguments as in the proof of Proposition B.1 of \cite{Drees.2015}.
	
	It remains to prove that the following condition holds:\\[0.5ex]	
	{\bf (D6)} \parbox[t]{10cm}{\begin{equation*}
		\lim_{\delta\downarrow 0} \limsup_{n\to\infty}P^*\Big\{ \int_0^\delta
		\sqrt{ \log N(\eps_0,\Phi,d_n)}\, d\eps_0 > \tau \Big\} = 0 \quad
		\forall \tau>0.
		\end{equation*}}
	
	Here $P^*$ denotes he outer probability, and the (random) covering number \linebreak $N(\eps_0,\Phi,d_n)$ is the minimum number of balls with radius $\eps_0$
	w.r.t.\ $$ d_n(\psi_1,\psi_2) := \bigg( \frac 1{nv_n} \sum_{j=1}^{m_n}
	\Big(\sum_{i=1}^{r_n}\psi_1(T_{n,j,i}^*)-\psi_2(T_{n,j,i}^*\Big)^2\bigg)^{1/2},
	$$
	needed to cover $\Phi$, $T_{n,j}^*=(T_{n,j,i}^*)_{1\le i\le r_n}$ denote iid copies of $(X_{n,i})_{1\le i\le r_n}$ (defined in \eqref{eq:Xnidef}) and $m_n:=\lfloor n/r_n\rfloor$.
	
	Note that this condition can be verified separately for the sets of functions $\{\phi_{0,s}\mid{s\in[1-\eps,1+\eps]}\}$, $\{\phi_{1,s}\mid{s\in[1-\eps,1+\eps]}\}$, $\Phi_2^t:=\{\phi_{2,x,s}^t\, |\, x\in[x_0,\infty),s\in[1-\eps,1+\eps]\}$ and $\Phi_3^t := \{\phi_{3,y,s}^t\mid y\in[y_0,\infty),s\in[1-\eps,1+\eps]\}$ for $|t|\in\{1,\ldots,\tilde t\}$. For the former two function classes, (D6) easily follows from the linear order of the functions in $s$; see \citep[Example 3.8]{Drees.2010} for details. It is much more challenging to verify (D6) for the remaining two families.
	We give details of the proof only for the class $\Phi_3^t$, but the arguments readily carry over to $\Phi_2^t$.
	
	Since  this most crucial part of the proof is rather involved, we first give a brief outline. In a first step, we use VC theory to bound the covering number of the family of functions
	$f_{3,y,s}^{(r,t)}:\big([0,\infty)^{(2\tilde{t}+1)}\big)^r\to\R$,
	\begin{equation*}	f_{3,y,s}^{(r,t)}(z_1,\dots,z_r):=\sum_{i=1}^r\phi_{3,y,s}^t(z_i)=\sum_{i=1}^r \Big(\frac{z_{i,-t}}{z_{i,0}}\Big)^{\alpha}\1\{z_i\in V_{y,s}^t\},
	\end{equation*}
	with $z_i=(z_{i,-\tilde{t}},\dots,z_{i,\tilde{t}})$ for fixed $r$ (see \eqref{eq:F3rcover}). Then we show that by choosing $r$ equal to the (random) minimum number (denoted by $R_{n,\eps_0}$) such that the maximal possible contribution of clusters of length larger than $r$ to the distance $d_n$ is less than $\eps_0/2$, one can bound the covering number $N(\eps_0,\Phi_3^t,d_n)$ by some function of $R_{n,\eps_0}$ (cf.\ \eqref{eq:covnumbound}). Finally, Condition (D6) follows if one shows that, with large probability, $R_{n,\eps_0}$ grows only polynomially in $\eps_0^{-1}$.
	
	First note that
	the function $\phi_{3,y,s}^t$ does not vanish on the set $V_{y,s}^t:=$ $\{(x_{-\tilde{t}},\dots,x_{\tilde{t}})\in[0,\infty)^{2\tilde{t}+1}\ |\ x_0/x_{-t}>y,\ x_{-t}>0,\ x_0>s\}$. We now show that the subgraphs $M_{y,s}^{(r,t)}:=\big\{(\lambda,z_1,\dots,z_r)\in\R\times([0,\infty)^{(2\tilde{t}+1)})^r\ |\ $ $\lambda<f_{3,y,s}^{(r,t)}(z_1,\dots,z_r)\big\}$ of $f_{3,y,s}^{(r,t)}$ form a VC class. To this end, consider an arbitrary set $A=\{(\lambda^{(l)},x_1^{(l)},\dots,x_r^{(l)})\ |\ 1\leqslant l\leqslant m\}\subset\R\times[0,\infty)^{r(2\tilde{t}+1)}$ of $m$ points. For $1\leqslant i\leqslant r$, $1\leqslant l\leqslant m$, define lines \linebreak
	$\{(x_{i,0}^{(l)}/x_{i,-t}^{(l)},s)\ |\ s\in[1-\eps,1+\eps]\}$ and $\{(y,x_{i,0}^{(l)})\ |\ y\in[y_0,\infty)\}$, that divide the set $[y_0,\infty)\times[1-\eps,1+\eps]$ into at most $(mr+1)^2$ rectangles. If $(y,s),(\tilde{y},\tilde{s})\in[y_0,\infty)\times[1-\eps,1+\eps]$ belong to the same rectangle then the symmetric difference $V_{y,s}^t\vartriangle V_{\tilde{y},\tilde{s}}^t$ does not contain any of the points $x_i^{(l)}$, $1\leqslant i\leqslant r$, $1\leqslant l\leqslant m$. Hence, the equality $f_{3,y,s}^{(r,t)}(x_1^{(l)},\dots,x_r^{(l)})=f_{3,\tilde{y},\tilde{s}}^{(r,t)}(x_1^{(l)},\dots,x_r^{(l)})$ holds for all $1\leqslant l\leqslant m$, and the intersections $A\cap M_{y,s}^t$ and $A\cap M_{\tilde{y},\tilde{s}}^t$ are identical. Thus, $(M_{y,s}^t)_{y\in[y_0,\infty),s\in[1-\eps,1+\eps]}$ can pick at most $(mr+1)^2$ different subset of $A$. If $m>4\log r$ and $r$ is sufficiently large then
	$m-2\log_2 m>3 \log r>\log_2(4r^2)$ which implies $2^m>4m^2r^2\geqslant (mr+1)^2$. Hence, the family of subgraphs $(M_{y,s}^t)_{y\in[y_0,\infty),s\in[1-\eps,1+\eps]}$ cannot shatter $A$, which shows that the VC-index of $\mathcal{F}_3^{(r,t)}:=\{f_{3,y,s}^{(r,t)}\ |\ y\in[y_0,\infty),s\in[1-\eps,1+\eps]\}$ is less than $4\log r$ if $r$ is sufficiently large. By Theorem 2.6.7 of \cite{Vaart.1996}, we have
	\begin{equation} \label{eq:F3rcover}
	N\Big(\delta\Big(\int G_r^2\ \mathrm{d}Q\Big)^{1/2},\mathcal{F}_3^{(r,t)},L_2(Q)\Big)\leqslant K_1r^{16}\delta^{-K_2\log r}
	\end{equation}
	for all small $\delta>0$, all probability measures $Q$ on $([0,\infty)^{(2\tilde{t}+1)})^r$  such that $\int G_r^2\ \mathrm{d}Q>0$, and suitable universal constants $K_1,K_2>0$ with $G_r=f_{3,y_0,1-\eps}^{(r,t)}$ denoting the envelope function of $\F_3^{(r,t)}$.
	
	In the second step we show that the terms pertaining to blocks with more than $r$ non-vanishing summands do not contribute too much to $d_n(\phi_{3,y,s}^t,\phi_{3,\tilde y,\tilde s}^t)$, if we let $r$ tend to infinity in a suitable way.
	
	Denote the number of independent blocks with at most $r$ non-zero entries by $N_{n,r}:=\sum_{j=1}^{m_n}\1\{H(T_{n,j}^*)\leqslant r\}$ with $H(z)=\sum_{i=1}^{r_n}\1\{z_{i,0}>1-\eps\}$, $z\in([0,\infty)^{(2\tilde{t}+1)})^{r_n}$.
	For these blocks, define vectors $\tilde T_{n,j}^*$ of length $r$ which contain all non-zero values of $T_{n,j}^*$, augmented by $r-H(T_{n,j}^*)$ zeros. Let
	\begin{equation*}
	Q_{n,r}:=\frac{1}{N_{n,r}}\sum_{j=1}^{m_n}\1\{H(T_{n,j}^*)\leqslant r\}\eps_{\tilde {T}_{n,j}^*},
	\end{equation*}
	with $\eps_T$ the Dirac measure with mass 1 in $T$.
	We can bound the squared distance between $\phi_{3,y,s}^t$ and $\phi_{3,\tilde y,\tilde s}^t$ as follows:
	
	\begin{align*}
	d_n^2(\phi_{3,y,s}^t,\phi_{3,\tilde y,\tilde s}^t)&=\frac{1}{nv_n}\sum_{j=1}^{m_n}\Big( \sum_{i=1}^{r_n}\phi_{3,y,s}^t(T_{n,j,i}^*)-\phi_{3,\tilde y,\tilde s}^t(T_{n,j,i}^*)\Big)^2\\
	&\leqslant \frac{N_{n,r}}{nv_n}\int(f_{3,y,s}^{(r,t)}-f_{3,\tilde y,\tilde s}^{(r,t)})^2\ \mathrm{d}Q_{n,r}+\frac{1}{nv_n}\sum_{j=1}^{m_n}G_{r_n}^2(T_{n,j}^*)\1\{H(T_{n,j}^*)>r\}
	\end{align*}
	for  all $r\in\N$. In particular,
	\begin{equation*}
	d_n^2(\phi_{3,y,s}^t,\phi_{3,\tilde y,\tilde s}^t)\leqslant \frac{N_{n,R_{n,\eps_0}}}{nv_n}\int (f_{3,y,s}^{(R_{n,\eps_0},t)}-f_{3,\tilde y,\tilde s}^{(R_{n,\eps_0},t)})^2\ \mathrm{d}Q_{n,R_{n,\eps_0}}+\frac{\eps_0^2}{2}
	\end{equation*}
	with
	\begin{equation*}
	R_{n,\eps_0}:=\min\Big\{r\in\N\ \Big|\ \frac{1}{nv_n}\sum_{j=1}^{m_n}G_{r_n}^2(T_{n,j}^*)\1\{H(T_{n,j}^*)>r\}<\frac{\eps_0^2}{2}\Big\}.
	\end{equation*}
	If $\int (f_{3,y,s}^{(R_{n,\eps_0},t)}-f_{3,\tilde y,\tilde s}^{(R_{n,\eps_0},t)})^2\ \mathrm{d}Q_{n,R_{n,\eps_0}}\leqslant nv_n\eps_0^2/(2N_{n,R_{n,\eps_0}})=:\eps_1^2$ then $d_n^2(\phi_{3,y,s}^t,\phi_{3,\tilde y,\tilde s}^t)\leqslant\eps_0^2$; that is, for vectors $(y,s),(\tilde y,\tilde s)$ such that $f_{3,y,s}^{(R_{n,\eps_0},t)}$ and $f_{3,\tilde y,\tilde s}^{(R_{n,\eps_0},t)}$ belong to some $\eps_1$-ball w.r.t.\ $L^2(Q_{n,R_{n,\eps_0}})$, the corresponding functions $\phi_{3,y,s}^t$ and $\phi_{3,\tilde y,\tilde s}^t$ belong to the same $\eps_0$-ball w.r.t.\ $d_n$. This implies $N(\eps_0,\Phi_3^t,d_n)\leqslant N(\eps_1,\mathcal{F}_3^{(R_{n,\eps_0},t)},L_2(Q_{n,R_{n,\eps_0}}))$.
	Note that
	\begin{equation*}
	\int G_{R_{n,\eps_0}}^2\ \mathrm{d}Q_{n,R_{n,\eps_0}} \leqslant \frac{y_0^{-2\alpha} R_{n,\eps_0}^2}{N_{n,R_{n,\eps_0}}} \sum_{j=1}^{m_n}\1\{H(T_{n,j}^*)> 0\}.
	\end{equation*}
	Using \eqref{eq:F3rcover}, we conclude that $\Phi_3^t$ can be covered by
	\begin{align*}
	N(\eps_0,\Phi_3^t,d_n)
	&\leqslant K_1R_{n,\eps_0}^{16}\bigg(\frac{\eps_1^2}{\int G_{R_{n,\eps_0}}^2 \mathrm{d}Q_{n,R_{n,\eps_0}}}\bigg)^{-K_2(\log R_{n,\eps_0})/2}\\
	&\leqslant K_1R_{n,\eps_0}^{16}\left( \frac{\eps_0y_0^{\alpha}}{ 2R_{n,\eps_0}}\left(\sum_{j=1}^{m_n}\frac{\1\{H(T_{n,j}^*)> 0\}}{2nv_n}\right)^{-\frac{1}{2}}\right)^{-K_2\log R_{n,\eps_0}}
	\end{align*}
	balls with radius $\eps_0$ w.r.t.\ $d_n$. Since by Chebyshev's inequality and regular variation
	\begin{equation*}
	P\Big\{\sum_{j=1}^{m_n}\1\{H(T_{n,j}^*)>0\}>2nv_n\Big\}\leqslant \frac{m_nr_nP\{X_0>(1-\eps)u_n\}}{(nv_n)^2}\to 0,
	\end{equation*}
	we conclude
	\begin{equation} \label{eq:covnumbound}
	N(\eps_0,\Phi_3^t,d_n)\leqslant K_1R_{n,\eps_0}^{16}\left(\frac{\eps_0y_0^\alpha}{2R_{n,\eps_0}}\right)^{-K_2\log R_{n,\eps_0}}
	\end{equation}
	with probability tending to $1$.
	
	It remains to show that $R_{n,\eps_0}$ does not increase too fast as $\eps_0$ tends to 0.
	To this end, we decompose the unit interval into intervals $(2^{-(l+1)}, 2^{-l}]$, $l\in\N_0$. Check that by Markov's inequality and Lemma \ref{lemma:CimpliesD}
	
	\begin{align*}
	P\Big\{ \frac{1}{nv_n}&\sum_{j=1}^{m_n}G_{r_n}^2(T_{n,j}^*)\1\{H(T_{n,j}^*)>M\eps_0^{-3}\}>\frac{\eps_0^2}{2}\text{ for some }0<\eps_0\leqslant 1\Big\}\\
	&\leqslant\sum_{l=0}^{\infty}P\Big\{\frac{1}{nv_n}\sum_{j=1}^{m_n}G_{r_n}^2(T_{n,j}^*)
	\1\{H(T_{n,j}^*)>M2^{3l}\}>\frac{2^{-2(l+1)}}{2}\Big\}\\
	&\leqslant\sum_{l=0}^{\infty}2^{2l+3}E\Big[\frac{y_0^{-2\alpha}}{nv_n}\sum_{j=1}^{m_n}H^2(T_{n,j}^*)
	\1\{H(T_{n,j}^*)>M2^{3l}\}\Big]\\
	&\leqslant \sum_{l=0}^{\infty}2^{2l+3}\frac{m_ny_0^{-2\alpha}}{nv_n}\frac{E[H^3(T_{n,1}^*)]}{M2^{3l}}\\
	&\leqslant \frac{K_3}M \sum_{l=0}^{\infty}2^{-l}\\
	&<\eta
	\end{align*}
	for some constant $K_3$ depending on $y_0$, and $M>2K_3/\eta$. Hence $R_{n,\eps_0}\leqslant M\eps_0^{-3}$ with probability greater than $1-\eta$, so that by \eqref{eq:covnumbound}
	\begin{equation*}
	\int_0^{\delta}(\log N(\eps_0,\Phi_3,d_n))^{1/2}\ \mathrm{d}\eps_0\leqslant \int_0^{\delta}(K_4+K_5|\log \eps_0|+K_6\log^2\eps_0)^{1/2}\ \mathrm{d}\eps_0
	\end{equation*}
	for suitable constants $K_4,K_5,K_6>0$. Now condition (D6) is obvious.
\end{proof}

\begin{Lem}\label{lemma:asympHill}
	If the conditions \eqref{eq:threshconsist}, A($x_0$), (B), (C) and \eqref{eq:biasHillcond} hold, then
	\begin{align*}
	(nv_n)^{1/2}(\hat{\alpha}_{n,\hat{u}_n}-\alpha)&
	\rightsquigarrow \alpha Z(\phi_{1,1})-\alpha^2Z(\phi_{0,1})
	\end{align*}
	where $Z$ is the same centered Gaussian process as in Proposition \ref{prop1}.
\end{Lem}

\begin{proof}
	For $s\in[1-\eps,1+\eps]$, define
	\begin{equation}\label{eq:hillestins}
	\alpha_{n,s}:=\frac{1}{E[\log(X_0/(su_n))\ |\ X_0>su_n]}\quad\text{and}\quad \tilde{\alpha}_{n,s}:=\frac{\sum_{i=1}^n\phi_{1,s}(X_{n,i})}{\sum_{i=1}^n\phi_{0,s}(X_{n,i})}
	\end{equation}
	and processes $V_n(s):=(nv_n)^{1/2}(\tilde{\alpha}_{n,s}-\alpha_{n,s})$ and $V(s):=s^{\alpha}(\alpha Z(\phi_{1,s})-\alpha^2 Z(\phi_{0,s}))$. Note that, by \eqref{eq:threshconsist}, $S_n\to 1$ in probability, and so $\hat{\alpha}_{n,\hat{u}_n}=\tilde{\alpha}_{n,S_n}$  with probability tending to 1. In view of \eqref{eq:biasHillcond}, $(nv_n)^{1/2}(\hat{\alpha}_{n,\hat{u}_n}-\alpha)=V_n(S_n)+o_P(1)$.
	
	By similar arguments as in proof of Lemma 4.4 of \cite{Drees.2015}, one can conclude from Proposition \ref{prop1} that
	$V_n\rightsquigarrow V$
	(w.r.t.\ the supremum norm) and that $V$ has continuous sample paths almost surely. Using Slutsky's lemma, we obtain $(V_n,S_n)\rightsquigarrow (V,1)$, and by Skorohod's theorem, there are versions for which the convergence holds almost surely.
	It follows that
	\begin{equation*}
	|V_n(S_n)-V(1)|\leqslant \sup_{s\in[1-\eps,1+\eps]}|V_n(s)-V(s)|+|V(S_n)-V(1)|\to 0
	\end{equation*}
	almost surely, from which the assertion is obvious.
\end{proof}

\begin{proofofthm}
	The assertion follows from arguments along the line of reasoning used  in the proof of Theorem 4.5 in \cite{Drees.2015} with similar modifications as employed in the proof of Lemma \ref{lemma:asympHill}.
\end{proofofthm}

\begin{proofoflemma}
	Fix an arbitrary $\delta\in(0,\eps)$ and choose some $a_{\delta}^+\in(0,1-(1+\delta)^{-\alpha})$ and $a_{\delta}^-\in(0,(1-\delta)^{-\alpha}-1)$. Then, by regular variation of $F^{\leftarrow}$, we have
	\begin{align*}
	(1+\delta)F^{\leftarrow}(1-k_n/n)&> F^{\leftarrow}(1-(1-a_{\delta}^+)k_n/n),\\
	(1-\delta)F^{\leftarrow}(1-k_n/n)&< F^{\leftarrow}(1-(1+a_{\delta}^-)k_n/n)
	\end{align*}
	for sufficiently large $n$, and hence $\bar F\big((1+\delta)F^\leftarrow(1-k_n/n)\big)<(1-a_\delta^+)k_n/n$ and $\bar F\big((1-\delta)F^\leftarrow(1-k_n/n)\big)>(1+a_\delta^-)k_n/n$. Let $u_n=F^{\leftarrow}(1-k_n/n)$ so that  $v_n=(1+o(1))k_n/n$ by the regular variation of $\bar F$ \citep[Th.\ 1.5.12]{Binghametal.1987}. The proof of Prop.~\ref{prop1} shows that $\tilde{Z}_n(\phi_{1,s})$ converge weakly to a normal distribution, because $\phi_{1,s}$ is almost surely continuous w.r.t.\ $\mathcal{L}(\bar{Y}_0)$.
	Thus,
	\begin{align*}
	&P\Big\{\frac{X_{n-k_n:n}}{F^{\leftarrow}(1-k_n/n)}>1+\delta\Big\}\\
	&\qquad=P\Big\{\sum_{i=1}^{n}\1\big\{X_i>(1+\delta)F^{\leftarrow}(1-k_n/n)\big\}>k_n\Big\}\\
	&\qquad=P\Big\{\tilde{Z}_n(\phi_{1,1+\delta})>(nv_n)^{-1/2}\big(k_n-nP\{X_0>(1+\delta)F^{\leftarrow}(1-k_n/n)\}\big)\Big\}\\
	&\qquad\leqslant P\Big\{\tilde{Z}_n(\phi_{1,1+\delta})>\frac{a_{\delta}^+}2 k_n^{1/2}\Big\}\to 0.
	\end{align*}
	Analogously, one obtains
	\begin{align*}
	P\Big\{\frac{X_{n-k_n:n}}{F^{\leftarrow}(1-k_n/n)}>1-\delta\Big\}\geqslant  P\Big\{\tilde{Z}_n(\phi_{1,1-\delta})>-a_{\delta}^-k_n^{1/2}\Big\}\to 1.
	\end{align*}
	Now let $\delta$ tend to 0 to conclude the assertion.
\end{proofoflemma}

\begin{proofofthmboot}
	First note that, by Corollary 3.9 of \cite{Drees.2010}, the asymptotic behavior of the processes $\tilde Z_n$ is not changed if one replaces $n$ with $r_nm_n$. One may easily conclude that, up to terms of the order $o_P((nv_n)^{-1/2})$, the estimators  $\hat F_{n,su_n}^{(f,\Theta_t)}$, $\hat F_{n,su_n}^{(b,\Theta_t)}$ and $\tilde\alpha_{n,s}$ (defined in \eqref{eq:hillestins}) do not change either. Hence, w.l.o.g.\ we assume that $n=m_nr_n$.
	
	The crucial observation to establish consistency of the bootstrap version is that the bootstrap processes
	$$ Z_n^*(\psi) := (nv_n)^{-1/2} \sum_{j=1}^{m_n} \xi_j \sum_{i\in I_j} \big(\psi(X_{n,i})-E\psi (X_{n,i})\big) $$
	converge to the same limit $Z$ as $Z_n$, both unconditionally and conditionally given $X_{n,1},\ldots,X_{n,n}$; see \cite{Drees.2015b}, Corollary 2.7.
	Define
	$$ \tilde{\alpha}_{n,s}^*:=\frac{\sum_{j=1}^{m_n}(1+\xi_j)\sum_{i\in I_j}\phi_{1,s}(X_{n,i})}{\sum_{j=1}^{m_n}(1+\xi_j)\sum_{i\in I_j}\phi_{0,s}(X_{n,i})}, $$
	and $V_n(s):=(nv_n)^{1/2}(\tilde{\alpha}_{n,s}^*-\tilde{\alpha}_{n,s})$. By similar calculations as in the proof of Theorem 3.3 of \cite{Davis.2018}, one obtains $V_n\rightsquigarrow V$ with $V(s):=s^{\alpha}(\alpha Z(\phi_{1,s})-\alpha^2 Z(\phi_{0,s}))$ denoting the limit process in Lemma  \ref{lemma:asympHill}. Since $V$ has a.s.\ continuous sample paths, $S_n\to 1$ in probability and $\hat\alpha_{n,\hat u_n}^*=\tilde{\alpha}_{n,S_n}^*$ and $\hat\alpha_{n,\hat u_n}=\tilde{\alpha}_{n,S_n}$ with probability tending to 1, the convergence $(nv_n)^{1/2}(\hat{\alpha}_{n,\hat u_n}^*-\hat{\alpha}_{n,\hat{u}_n})=V_n(S_n)+o_P(1)\rightsquigarrow V(1)$ follows readily.

	Now one may argue as in the proof of Theorem 3.3 of \cite{Davis.2018} to verify the assertion. To this end, one must replace $u_n$ with $\hat u_n=S_nu_{n}$ everywhere. For example, equation (6.10) of  \cite{Davis.2018} now becomes
	$$ \sum_{j=1}^{m_n} \xi_j\sum_{i\in I_j} \1\{X_i>S_nu_{n}\} = (nv_n)^{1/2} Z_{n,\xi}(\phi_{1,S_n}) + r_n\sum_{j=1}^{m_n} \xi_j P\{X_0>su_{n}\}|_{s=S_n}, $$
	where the last term is of stochastic order $r_nm_n^{1/2}v_n=o((nv_n)^{1/2})$. Thus
	\begin{equation}  \label{eq:bootex}
	\sum_{j=1}^{m_n} (1+\xi_j)\sum_{i\in I_j} \1\{X_i>S_nu_{n}\} = \sum_{i=1}^{n} \1\{X_i>S_nu_{n}\} +O_P\big((nv_n)^{1/2}\big).
	\end{equation}
	Since, by the law of large numbers, $\sum_{i=1}^{n} \1\{X_i>su_{n}\}=nP\{X_0>su_n\}(1+o_P(1))$ for all fixed $s\in [1-\eps,1+\eps]$ and $S_n\to 1$ in probability, a standard argument shows that the right hand side of \eqref{eq:bootex} equals $nv_n(1+o_P(1))$, and the proof can be concluded as in  \cite{Davis.2018}.
\end{proofofthmboot}

\noindent\textbf{Acknowledgements}\quad
	We thank the associate editor and two referees for their suggestions and comments which helped to improve the presentation of the results and to simplify the discussion of the stochastic recurrence equations.
\bibliography{Quelle}{}  
%
%
\end{document}


\maketitle

In this supplement, we present the explicit covariance function of the limit process in Theorem 2.1 and give further results from the simulation study in Section 3.

\section*{Covariances}
Using formula (C.1) of Proposition C.2, we will calculate the covariances of the limit process in Theorem 2.1 of the manuscript. Recall that the limit process is composed of
\begin{align*}
Z_{fw}^{(t)}(x)&:=Z(\phi_{2,x,1}^t)-\bar{F}^{(\Theta_t)}(x)Z(\phi_{1,1}),\\
Z_{bw}^{(t)}(y)&:=Z(\phi_{3,y,1}^t)-\bar{F}^{(\Theta_t)}(y)Z(\phi_{1,1})+Z_{\alpha}(y)
\end{align*}
with
\begin{equation*}
Z_{\alpha}(y)=(\alpha^2Z(\phi_{0,1})-\alpha Z(\phi_{1,1}))E[\log(\Theta_t)\1\{\Theta_t>y\}].
\end{equation*}
Let $x,x^*\in [x_0,\infty),$ $y,y^*\in [y_0,\infty),$ and $t,t^*\in\{-\tilde{t},\dots,\tilde{t}\}\setminus\{0\}$. Then, we have
\begin{align*}
\var (Z(\phi_{1,1}))&=1+2\sum_{k=1}^{\infty}P\{Y_0\Theta_k>1\},\\
\cov(Z(\phi_{1,1}),Z(\phi_{0,1}))&=\alpha^{-1}+\sum_{k=1}^{\infty}(E[\log(Y_0)\1\{Y_0\Theta_k>1\}]+E[\log^+(Y_0\Theta_k)]),\\
\cov(Z(\phi_{1,1}),Z(\phi_{2,x,1}^t))&=\bar{F}^{(\Theta_t)}(x)+\sum_{k=1}^{\infty}\Big(P\{\Theta_t>x,Y_0\Theta_k>1\}+P\Big\{\frac{\Theta_{k+t}}{\Theta_k}>x,Y_0\Theta_k>1\Big\}\Big),\\
\cov(Z(\phi_{1,1}),Z(\phi_{3,y,1}^t))&=E[\Theta_{-t}^{\alpha}\1\{\Theta_{-t}<y^{-1}\}]\\
&\quad+\sum_{k=1}^{\infty}\Big(E[\Theta_{-t}^{\alpha}\1\{\Theta_{-t}<y^{-1},Y_0\Theta_k>1\}]+E\Big[\Big(\frac{\Theta_{k-t}}{\Theta_k}\Big)^{\alpha}\1\Big\{\frac{\Theta_{k-t}}{\Theta_k}<y^{-1},Y_0\Theta_k>1\Big\}\Big]\Big),\\
\var(Z(\phi_{0,1}))&=2\alpha^{-2}+2\sum_{k=1}^{\infty}E[\log(Y_0)\log^+(Y_0\Theta_k)],\\
\cov(Z(\phi_{0,1}),Z(\phi_{2,x,1}^t))&=\alpha^{-1}\bar{F}^{(\Theta_t)}(x)\\
&\quad+\sum_{k=1}^{\infty}\Big(E[\log^+(Y_0\Theta_k)\1\{\Theta_t>x,Y_0\Theta_k>1\}]+E\Big[\log(Y_0)\1\Big\{\frac{\Theta_{k+t}}{\Theta_k}>x,Y_0\Theta_k>1\Big\}\Big]\Big),\\
\cov(Z(\phi_{0,1}),Z(\phi_{3,y,1}^t))&=\alpha^{-1}E[\Theta_{-t}^{\alpha}\1\{\Theta_{-t}<y^{-1}\}]+\sum_{k=1}^{\infty}\Big(E[\Theta_{-t}^{\alpha}\log^+(Y_0\Theta_k)\1\{\Theta_{-t}<y^{-1},Y_0\Theta_k>1\}]\\
&\quad+E\Big[\Big(\frac{\Theta_{k-t}}{\Theta_k}\Big)^{\alpha}\log(Y_0)\1\Big\{\frac{\Theta_{k-t}}{\Theta_k}<y^{-1},Y_0\Theta_k>1\Big\}\Big]\Big),\\
\cov(Z(\phi_{2,x,1}^t),Z(\phi_{2,x^*,1}^{t^*}))&=P\{\Theta_t>x,\Theta_{t^*}>x^*\}\\
&\quad+\sum_{k=1}^{\infty}\Big(P\Big\{\Theta_t>x,\frac{\Theta_{k+t^*}}{\Theta_k}>x^*,Y_0\Theta_k>1\Big\}+P\Big\{\frac{\Theta_{k+t}}{\Theta_k}>x,\Theta_{t^*}>x^*,Y_0\Theta_k>1\Big\}\Big),
\end{align*}
\pagebreak
\begin{align*}
\cov(Z(\phi_{2,x,1}^t),Z(\phi_{3,y^*,1}^{t^*}))&=E[\Theta_{-t^*}^{\alpha}\1\{\Theta_{-t^*}<(y^*)^{-1},\Theta_t>x\}]\\
&\quad+\sum_{k=1}^{\infty}\Big(E\Big[\Theta_{-t^*}^{\alpha}\1\Big\{\Theta_{-t^*}<(y^*)^{-1},\frac{\Theta_{k+t}}{\Theta_k}>x,Y_0\Theta_k>1\Big\}\Big]\\
&\quad\qquad\qquad+E\Big[\Big(\frac{\Theta_{k-t^*}}{\Theta_k}\Big)^{\alpha}\1\Big\{\frac{\Theta_{k-t^*}}{\Theta_k}<(y^*)^{-1},\Theta_{t}>x,Y_0\Theta_k>1\Big\}\Big]\Big),\\
\cov(Z(\phi_{3,y,1}^t),Z(\phi_{3,y^*,1}^{t^*}))&=E[\Theta_{-t}^{\alpha}\Theta_{-t^*}^{\alpha}\1\{\Theta_{-t}<y^{-1},\Theta_{-t^*}<(y^*)^{-1}\}]\\
&\quad+\sum_{k=1}^{\infty}\Big(E\Big[\Theta_{-t}^{\alpha}\Big(\frac{\Theta_{k-t^*}}{\Theta_k}\Big)^{\alpha}\1\Big\{\Theta_{-t}<y^{-1},\frac{\Theta_{k-t^*}}{\Theta_k}<(y^*)^{-1},Y_0\Theta_k>1\Big\}\Big]\\
&\quad\qquad\qquad+E\Big[\Big(\frac{\Theta_{k-t}}{\Theta_k}\Big)^{\alpha}\Theta_{-t^*}^{\alpha}\1\Big\{\frac{\Theta_{k-t}}{\Theta_k}<y^{-1},\Theta_{-t^*}<(y^*)^{-1},Y_0\Theta_k>1\Big\}\Big]\Big).
\end{align*}
The expected values containing $Y_0$ can also be expressed in terms of the spectral process only, using that $Y_0$ is standard Pareto random variable independent of the spectral process $(\Theta_k)_{k\in\Z}$. For example, with Fubini's theorem we obtain
\begin{align*}
E[\log^+(Y_0\Theta_k)]&= E\Big[\alpha \int_1^{\infty}\log^+(y\Theta_k)y^{-(\alpha+1)}\ \mathrm{d}y\Big]=E\Big[-y^{-\alpha}(\log(y\Theta_k)+\alpha^{-1})\Big|_{\max\{1,1/\Theta_k\}}^{\infty}\Big]\\
&=E\big[\min\{1,\Theta_k^{\alpha}\}\big(\log^+\Theta_k+\alpha^{-1}\big)\big]=E[\log^+\Theta_k]+\alpha^{-1}E[\min\{1,\Theta_k^{\alpha}\}].
\end{align*}

The covariance function of the limit process can be written as follows:
\begin{align*}
\cov(Z_{fw}^{(t)}(x),Z_{fw}^{(t^*)}(x^*))&=P\{\Theta_t>x,\Theta_{t^*}>x^*\}-\bar{F}^{(\Theta_t)}(x)\bar{F}^{(\Theta_{t^*})}(x^*)\\
&\quad+\sum_{k=1}^{\infty}\Big(P\Big\{\Theta_t>x,\frac{\Theta_{k+t^*}}{\Theta_k}>x^*,Y_0\Theta_k>1\Big\}+P\Big\{\frac{\Theta_{k+t}}{\Theta_k}>x,\Theta_{t^*}>x^*,Y_0\Theta_k>1\Big\}\Big)\\
&\quad-\bar{F}^{(\Theta_{t^*})}(x^*)\sum_{k=1}^{\infty}\Big(P\{\Theta_t>x,Y_0\Theta_k>1\}+P\Big\{\frac{\Theta_{k+t}}{\Theta_k}>x,Y_0\Theta_k>1\Big\}\Big)\\
&\quad-\bar{F}^{(\Theta_{t})}(x)\sum_{k=1}^{\infty}\Big(P\{\Theta_{t^*}>x^*,Y_0\Theta_k>1\}+P\Big\{\frac{\Theta_{k+t^*}}{\Theta_k}>x^*,Y_0\Theta_k>1\Big\}\Big),
\end{align*}
\begin{align*}
&\cov(Z_{fw}^{(t)}(x),Z_{bw}^{(t^*)}(y^*))\\
&=E[\Theta_{-t^*}^{\alpha}\1\{\Theta_{-t^*}<(y^*)^{-1},\Theta_t>x\}]\\
&\qquad+\sum_{k=1}^{\infty}\Big(E\Big[\Theta_{-t^*}^{\alpha}\1\Big\{\Theta_{-t^*}<(y^*)^{-1},\frac{\Theta_{k+t}}{\Theta_k}>x,Y_0\Theta_k>1\Big\}\Big]\\
&\qquad\hspace{30ex}+E\Big[\Big(\frac{\Theta_{k-t^*}}{\Theta_k}\Big)^{\alpha}\1\Big\{\frac{\Theta_{k-t^*}}{\Theta_k}<(y^*)^{-1},\Theta_t>x,Y_0\Theta_k>1\Big\}\Big]\Big)\\
&\qquad-\bar{F}^{(\Theta_{t^*})}(y^*)\sum_{k=1}^{\infty}\Big(P\{\Theta_t>x,Y_0\Theta_k>1\}+P\Big\{\frac{\Theta_{k+t}}{\Theta_k}>x,Y_0\Theta_k>1\Big\}-2\bar{F}^{(\Theta_{t})}(x)P\{Y_0\Theta_k>1\}\Big)\\
&\qquad-\bar{F}^{(\Theta_{t})}(x)\bigg[E[\Theta_{-t^*}^{\alpha}\1\{\Theta_{-t^*}<(y^*)^{-1}\}]\\
&\qquad\qquad+\sum_{k=1}^{\infty}\Big(E[\Theta_{-t^*}^{\alpha}\1\{\Theta_{-t^*}<(y^*)^{-1},Y_0\Theta_k>1\}]+E\Big[\Big(\frac{\Theta_{k-t^*}}{\Theta_k}\Big)^{\alpha}\1\Big\{\frac{\Theta_{k-t^*}}{\Theta_k}<(y^*)^{-1},Y_0\Theta_k>1\Big\}\Big]\Big)\bigg]\\
&\qquad+E[\log(\Theta_{t^*})\1\{\Theta_{t^*}>y^*\}]\\
&\qquad\qquad\times\bigg[\alpha^2\Big(\sum_{k=1}^{\infty}\Big(E[\log(Y_0\Theta_k)\1\{\Theta_t>x,Y_0\Theta_k>1\}]+E\Big[\log(Y_0)\1\Big\{\frac{\Theta_{k+t}}{\Theta_k}>x,Y_0\Theta_k>1\Big\}\Big]\Big)\Big)\\
&\qquad\qquad\qquad -\alpha\Big(\bar{F}^{(\Theta_t)}(x)+\sum_{k=1}^{\infty}\Big(P\{\Theta_t>x,Y_0\Theta_k>1\}+P\Big\{\frac{\Theta_{k+t}}{\Theta_k}>x,Y_0\Theta_k>1\Big\}\Big)\Big)\bigg]\\
&\qquad-\bar{F}^{(\Theta_t)}(x)E[\log(\Theta_{t^*})\1\{\Theta_{t^*}>y^*\}]\\
&\qquad\qquad\times\bigg[\alpha^2\Big(\sum_{k=1}^{\infty}\Big(E[\log(Y_0)\1\{Y_0\Theta_k>1\}]+E[\log^+(Y_0\Theta_k)]\Big)\Big)-\alpha\Big(1+2\sum_{k=1}^{\infty}P\{Y_0\Theta_k>1\}\Big)\bigg]
\end{align*}
\pagebreak
and
\begin{align*}
&\cov(Z_{bw}^{(t)}(y),Z_{bw}^{(t^*)}(y^*))\\
&=E[\Theta_{-t}^\alpha\Theta_{-t^*}^{\alpha}\1\{\Theta_{-t}<y^{-1},\Theta_{-t^*}<(y^*)^{-1}\}]\\
&\qquad+\sum_{k=1}^{\infty}\Big(E\Big[\Theta_{-t}^{\alpha}\Big(\frac{\Theta_{k-t^*}}{\Theta_k}\Big)^{\alpha}\1\Big\{\Theta_{-t}<y^{-1},\frac{\Theta_{k-t^*}}{\Theta_k}<(y^*)^{-1},Y_0\Theta_k>1\Big\}\Big]\\
&\qquad\qquad\qquad+E\Big[\Big(\frac{\Theta_{k-t}}{\Theta_k}\Big)^{\alpha}\Theta_{-t^*}^{\alpha}\1\Big\{\frac{\Theta_{k-t}}{\Theta_k}<y^{-1},\Theta_{-t^*}<(y^*)^{-1},Y_0\Theta_k>1\Big\}\Big]\Big)\\
&\qquad-\big(\bar{F}^{(\Theta_{t^*})}(y^*)+\alpha E[\log(\Theta_{t^*})\1\{\Theta_{t^*}>y^*\}]\big)\Big(E[\Theta_{-t}^{\alpha}\1\{\Theta_{-t}<y^{-1}\}]\\
&\hspace{14ex}+\sum_{k=1}^{\infty}\Big(E[\Theta_{-t}^{\alpha}\1\{\Theta_{-t}<y^{-1},Y_0\Theta_k>1\}]+E\Big[\Big(\frac{\Theta_{k-t}}{\Theta_k}\Big)^{\alpha}\1\Big\{\frac{\Theta_{k-t}}{\Theta_k}<y^{-1},Y_0\Theta_k>1\Big\}\Big]\Big)\Big)\\
&\qquad-\big(\bar{F}^{(\Theta_{t})}(y)+\alpha E[\log(\Theta_t)\1\{\Theta_t>y\}]\big)\Big(E[\Theta_{-t^*}^{\alpha}\1\{\Theta_{-t^*}<(y^*)^{-1}\}]\\
&\hspace{14ex}+\sum_{k=1}^{\infty}\Big(E[\Theta_{-t^*}^{\alpha}\1\{\Theta_{-t^*}<(y^*)^{-1},Y_0\Theta_k>1\}]+E\Big[\Big(\frac{\Theta_{k-t^*}}{\Theta_k}\Big)^{\alpha}\1\Big\{\frac{\Theta_{k-t^*}}{\Theta_k}<(y^*)^{-1},Y_0\Theta_k>1\Big\}\Big]\Big)\Big)\\
&\qquad + \bar{F}^{(\Theta_t)}(y)\bar{F}^{(\Theta_{t^*})}(y^*)\Big(1+2\sum_{k=1}^{\infty}P\{Y_0\Theta_k>1\}\Big)\\
&\qquad + E[\log(\Theta_{t^*})\1\{\Theta_{t^*}>y^*\}]\bigg[\alpha^2\Big(\alpha^{-1}E[\Theta_{-t}^{\alpha}\1\{\Theta_{-t}<y^{-1}\}]\\
&\qquad\qquad +\sum_{k=1}^{\infty}\Big(E[\Theta_{-t}^{\alpha}\log(Y_0\Theta_k)\1\{\Theta_{-t}<y^{-1},Y_0\Theta_k>1\}]+E\Big[\Big(\frac{\Theta_{k-t}}{\Theta_k}\Big)^{\alpha}\log(Y_0)\1\Big\{\frac{\Theta_{k-t}}{\Theta_k}<y^{-1},Y_0\Theta_k>1\}\Big]\Big)\Big)\bigg]\\
&\qquad +E[\log(\Theta_t)\1\{\Theta_t>y\}]\bigg[\alpha^2\Big(\alpha^{-1}E[\Theta_{-t^*}^{\alpha}\1\{\Theta_{-t^*}<(y^*)^{-1}\}]\\
&\qquad\qquad +\sum_{k=1}^{\infty}\Big(E[\Theta_{-t^*}^{\alpha}\log(Y_0\Theta_k)\1\{\Theta_{-t^*}<(y^*)^{-1},Y_0\Theta_k>1\}]\\
&\qquad\hspace{30ex}+E\Big[\Big(\frac{\Theta_{k-t^*}}{\Theta_k}\Big)^{\alpha}\log(Y_0)\1\Big\{\frac{\Theta_{k-t^*}}{\Theta_k}<(y^*)^{-1},Y_0\Theta_k>1\}\Big]\Big)\Big)\bigg]\\
&\qquad-\bigg[\bar{F}^{(\Theta_t)}(y)E[\log(\Theta_{t^*})\1\{\Theta_{t^*}>y^*\}]+\bar{F}^{(\Theta_{t^*})}(y^*)E[\log(\Theta_{t})\1\{\Theta_{t}>y\}]\\
&\qquad\qquad +2\alpha E[\log(\Theta_t)\log(\Theta_{t^*})\1\{\Theta_t>y\}\1\{\Theta_{t^*}>y^*\}]\bigg]\times\bigg[\alpha^2\Big(\alpha^{-1}\\
&\qquad\qquad\qquad+\sum_{k=1}^{\infty}(E[\log(Y_0)\1\{Y_0\Theta_k>1\}]+E[\log^+(Y_0\Theta_k)])\Big)-\alpha\Big(1+2\sum_{k=1}^{\infty}P\{Y_0\Theta_k>1\}\Big)\bigg]\\
&\qquad+E[\log(\Theta_t)\log(\Theta_{t^*})\1\{\Theta_t>y\}\1\{\Theta_{t^*}>y^*\}]\\
&\qquad\quad\times \bigg[\alpha^4\Big(2\alpha^{-2}+2\sum_{k=1}^{\infty}E[\log(Y_0)\log^+(Y_0\Theta_k)]\Big)-\alpha^2\Big(1+2\sum_{k=1}^{\infty}P\{Y_0\Theta_k>1\}\Big)\bigg].
\end{align*}
\pagebreak

\section*{Further simulation results}

Here, we compare the distribution of our estimators  of $P\{\Theta_t>x\}$ based on exceedances over quantiles $F^{\leftarrow}(1-k/n)$ and over the corresponding order statistics $X_{n-k:n}$, respectively.
In the nGARCH, tGARCH, gumCopula (with $\theta\in\{1.2,1.5,2\}$) and tCopula (with $\rho\in\{0.25,0.5, 0.75\}$) models described in Section 3 of the manuscript, we consider lags $t\in\{1,5\}$ and arguments $x\in\{1/2,1\}$. The sample sizes, number of simulations and the lay-out of figures are the same as in the manuscript. In each figure, Q-Q plots of the forward estimator are shown in the top row, while the bottom row shows results for backward estimator. In the left column, the theoretical resp.\ empirical quantile to the 90\% level is used as threshold, and the right plots correspond to the 95\% level. The main diagonal is indicated by a dashed red line.

Figures \ref{fig1} and \ref{fig2} show the Q-Q plots for the estimators of $P\{\Theta_5>1\}$ resp. $P\{\Theta_5>1/2\}$, i.e.\ they are counterparts to Figures 1 and 3 of the manuscript for lag 5 instead of 1. Indeed, the corresponding plots look very similar. In particular, most points lie quite close to the main diagonal, but the version based on exceedances over order statistics performs a bit better in that extreme estimation errors occur more rarely.

\begin{figure}[b]
	\centering
	\includegraphics[width=1\textwidth]{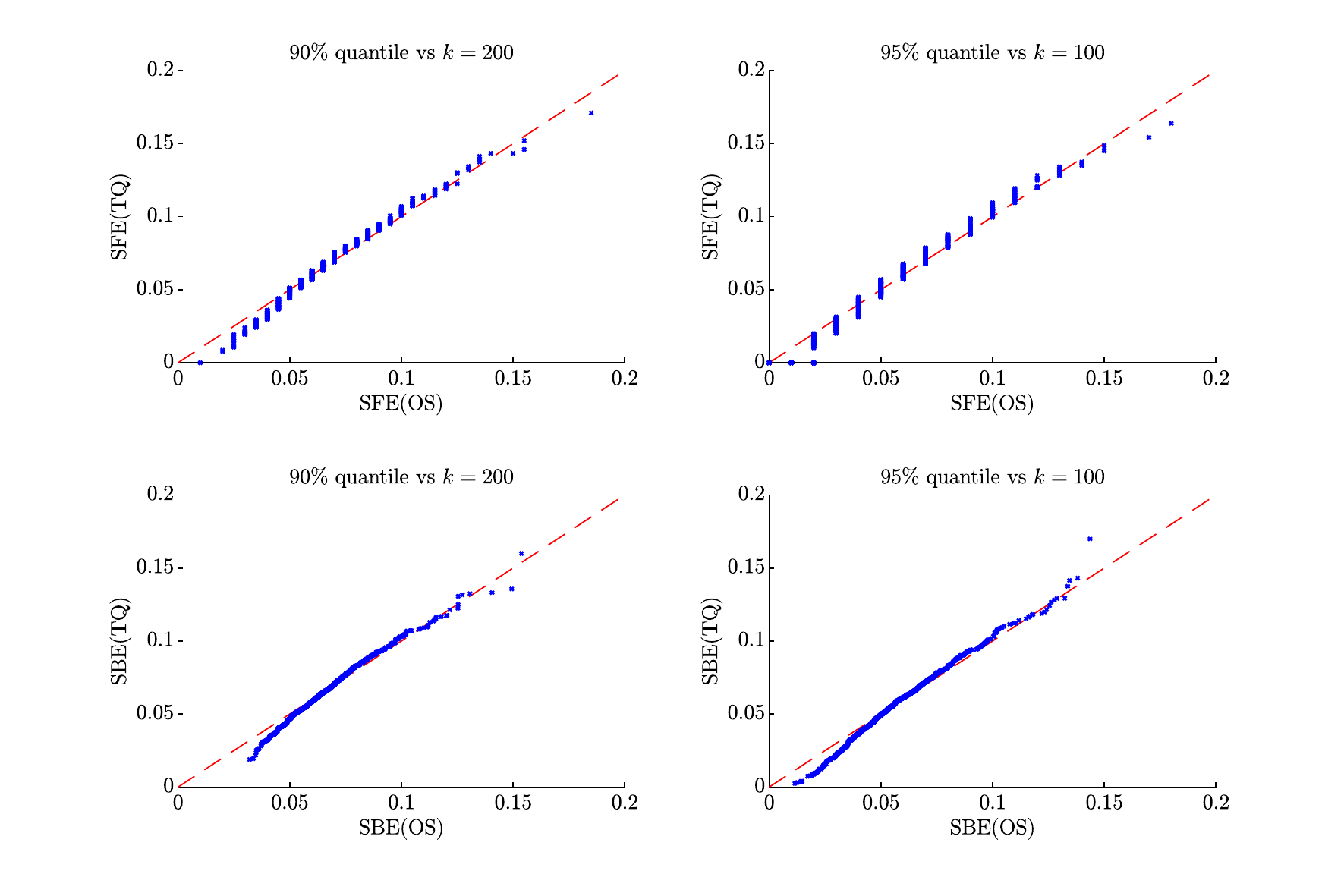}
	\caption{Comparison of estimators of  $P\{\Theta_5>1\}$  in the nGARCH model}
	\label{fig1}
\end{figure}
\begin{figure}[tb]
	\centering
	\includegraphics[width=1\textwidth]{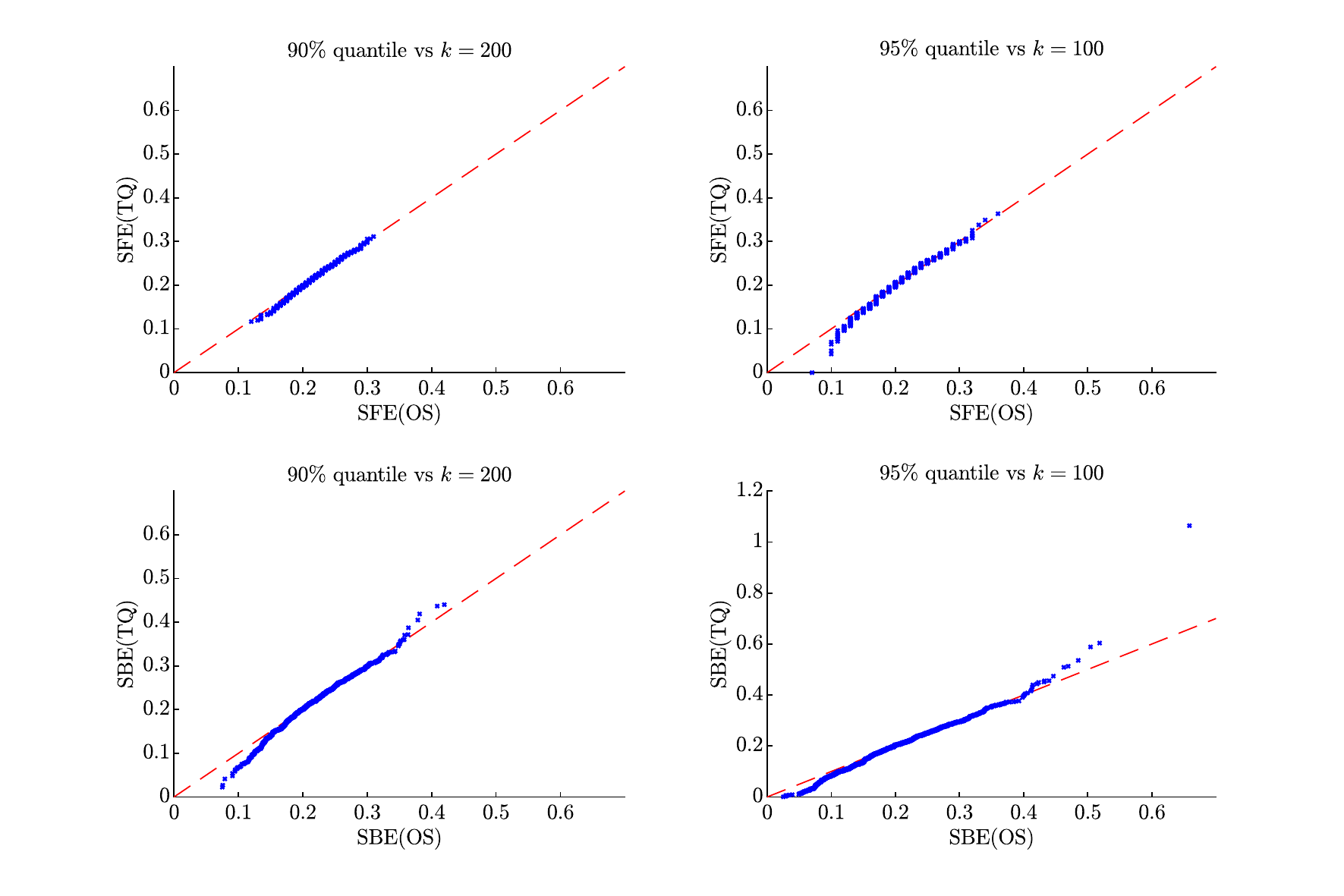}
	\caption{ Comparison of estimators of  $P\{\Theta_5>1/2\}$  in the nGARCH model}
	\label{fig2}
\end{figure}

Figure \ref{fig3}--\ref{fig6} show the Q-Q plots in the tGARCH model with $x\in\{1/2,1\}$ and lags $t\in\{1,5\}$. Overall, the plots are similar to the ones for the nGARCH model, but the strong overestimation by the backward estimator based on exceedances over the true quantile, which occurred in a few simulations in the nGARCH model, is not observed here.

\begin{figure}[tb]
	\centering
	\includegraphics[width=1\textwidth]{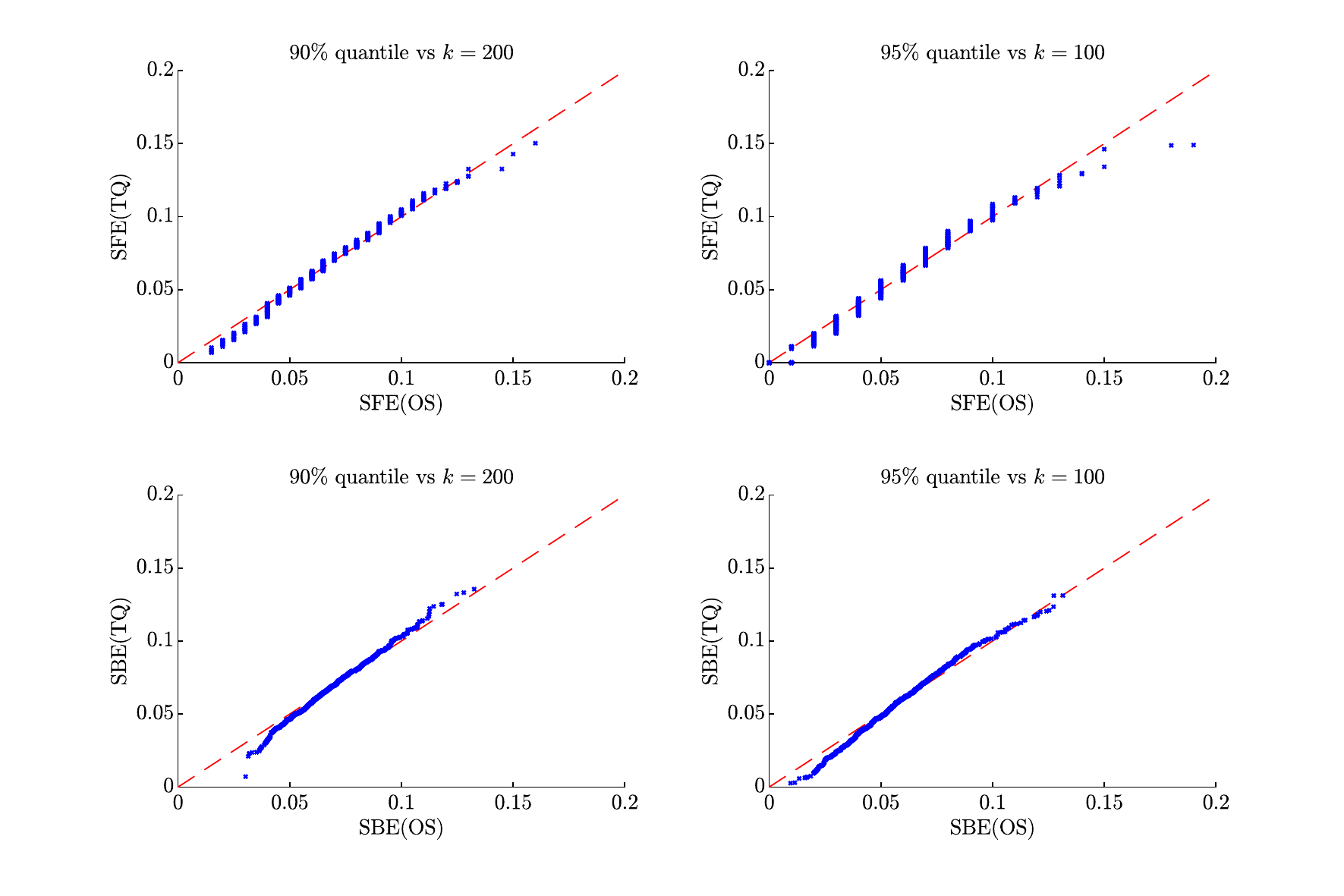}
	\caption{ Comparison of estimators of  $P\{\Theta_1>1\}$  in the tGARCH model}
	\label{fig3}
\end{figure}
\begin{figure}[tb]
	\centering
	\includegraphics[width=1\textwidth]{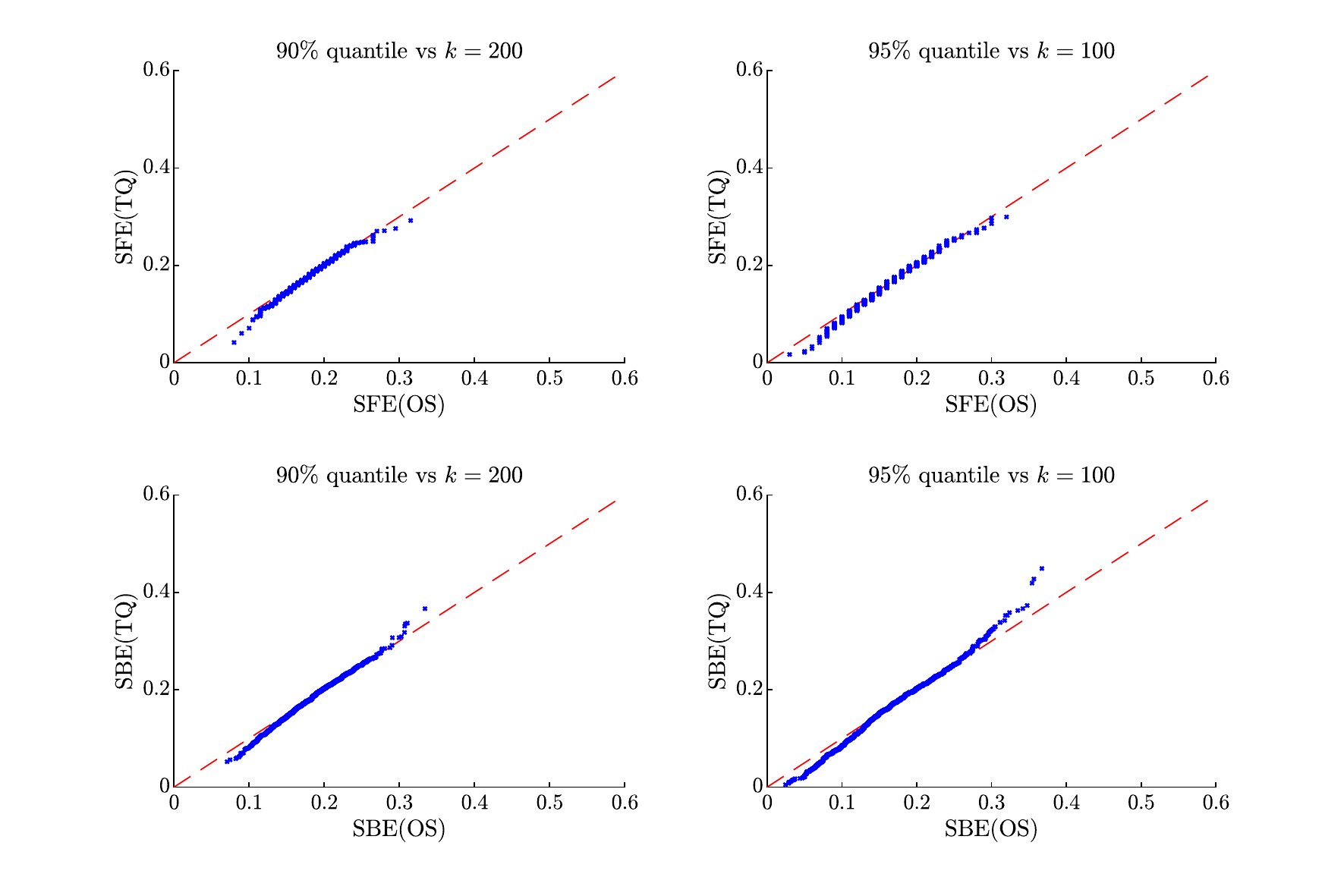}
	\caption{Comparison of estimators of  $P\{\Theta_1>1/2\}$  in the tGARCH model}
	\label{fig4}
\end{figure}
\begin{figure}[tb]
	\centering
	\includegraphics[width=1\textwidth]{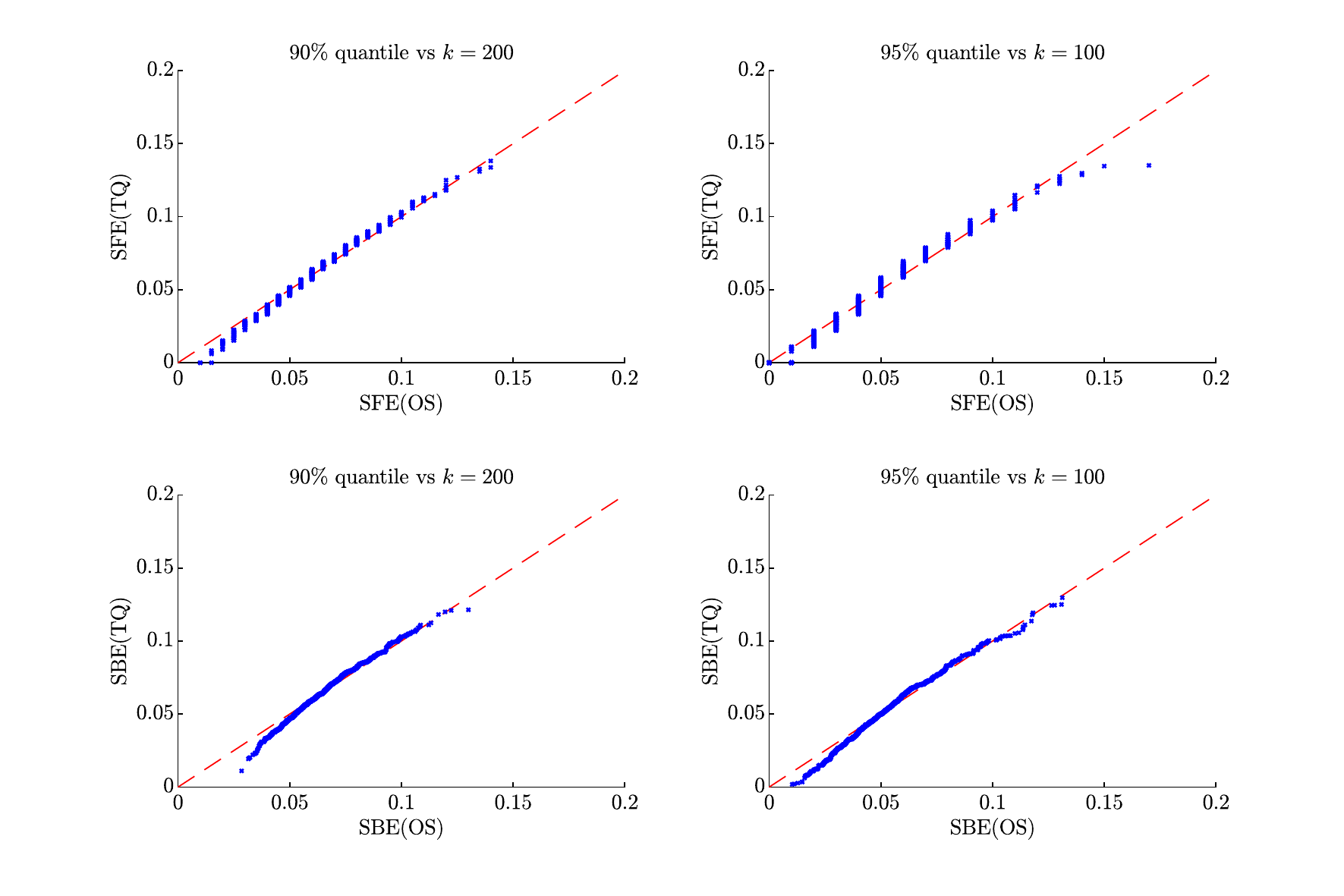}
	\caption{Comparison of estimators of  $P\{\Theta_5>1\}$  in the tGARCH model}
	\label{fig5}
\end{figure}

\begin{figure}[tb]
	\centering
	\includegraphics[width=1\textwidth]{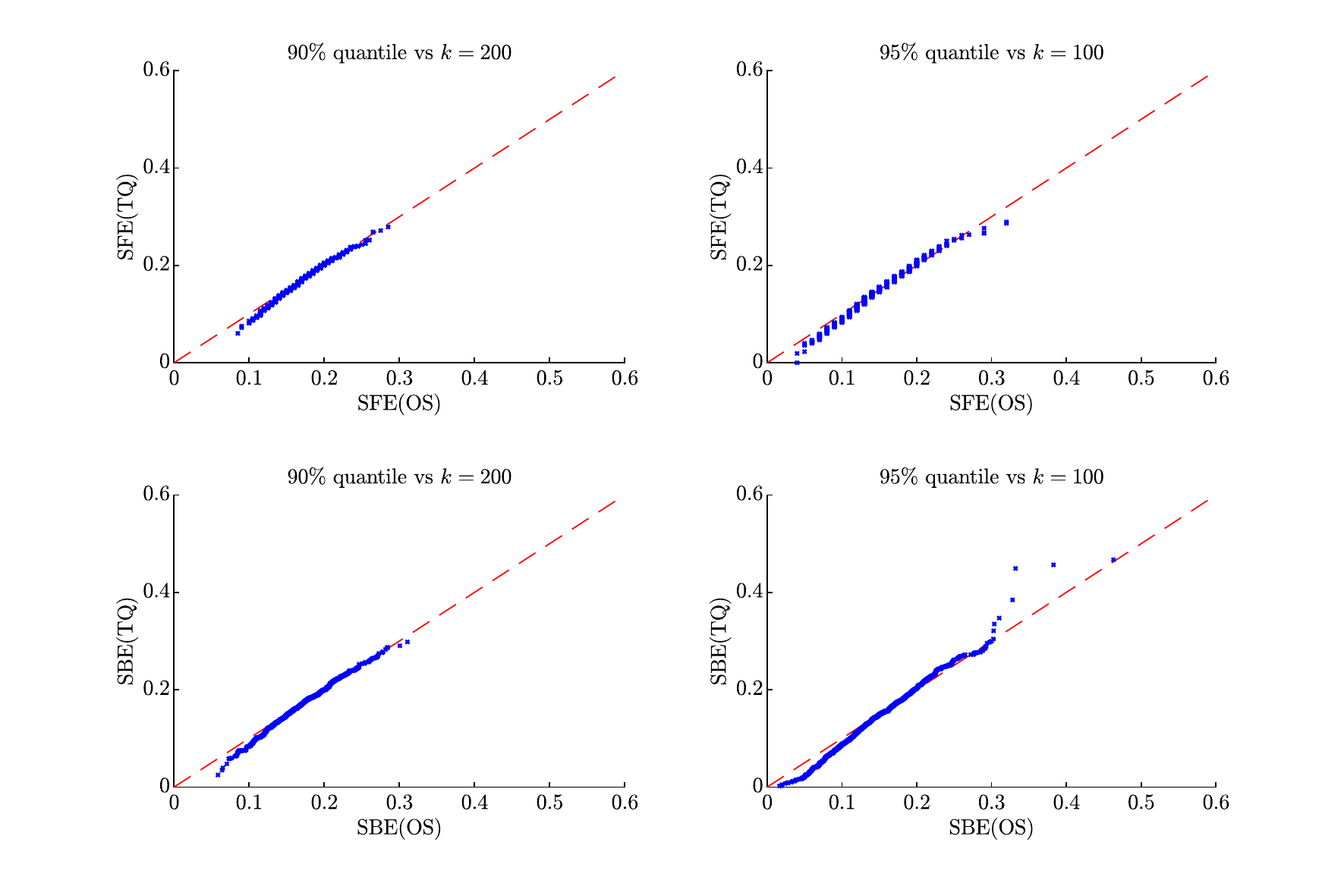}
	\caption{Comparison of estimators of  $P\{\Theta_5>1/2\}$  in the tGARCH model}
	\label{fig6}
\end{figure}

Figures \ref{figtcop0} and \ref{figtcop1} compare both versions of the estimators of the probabilities $P\{\Theta_5>1/2\}$  and $P\{\Theta_5>1\}$, respectively, in the tCopula model with $\rho=0.25$, while
Figures \ref{figtcop1}--\ref{figtcop3} display Q-Q plots for the estimators of $P\{\Theta_5>1\}$ in the tCopula model with $\rho=0.25$, $\rho=0.5$ and $\rho=0.75$, respectively, that is, in models with increasing serial dependence.  The difference between the distributions of both versions is even smaller than in the GARCH models. In fact, in all settings, the distributions of both versions of the backward estimator are almost identical, whereas, roughly speaking, the forward estimator based on exceedances over order statistics is a discretized version of the TQ-version. Note that the discrete nature of the distribution of the OS-version of the forward estimator becomes the more pronounced the weaker the dependence is (and hence the smaller the probabilities to be estimated) since then there are fewer values the estimator can attain near the true probability.
\begin{figure}[tb]
	\centering
	\includegraphics[width=1\textwidth]{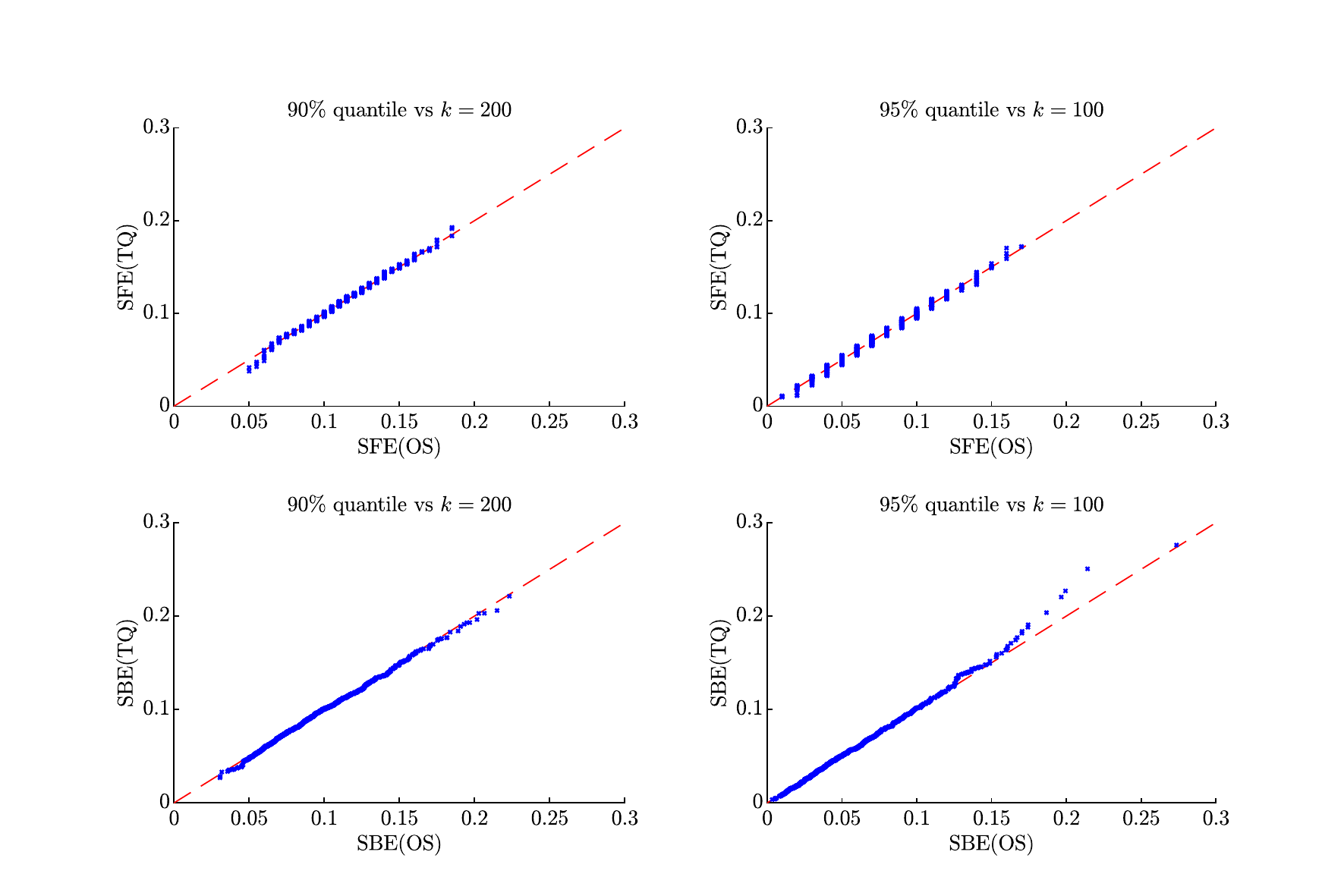}
	\caption{Comparison of estimators of  $P\{\Theta_5>1/2\}$  in the tCopula model  with $\rho=0.25$}
	\label{figtcop0}
\end{figure}
\begin{figure}[tb]
	\centering
	\includegraphics[width=1\textwidth]{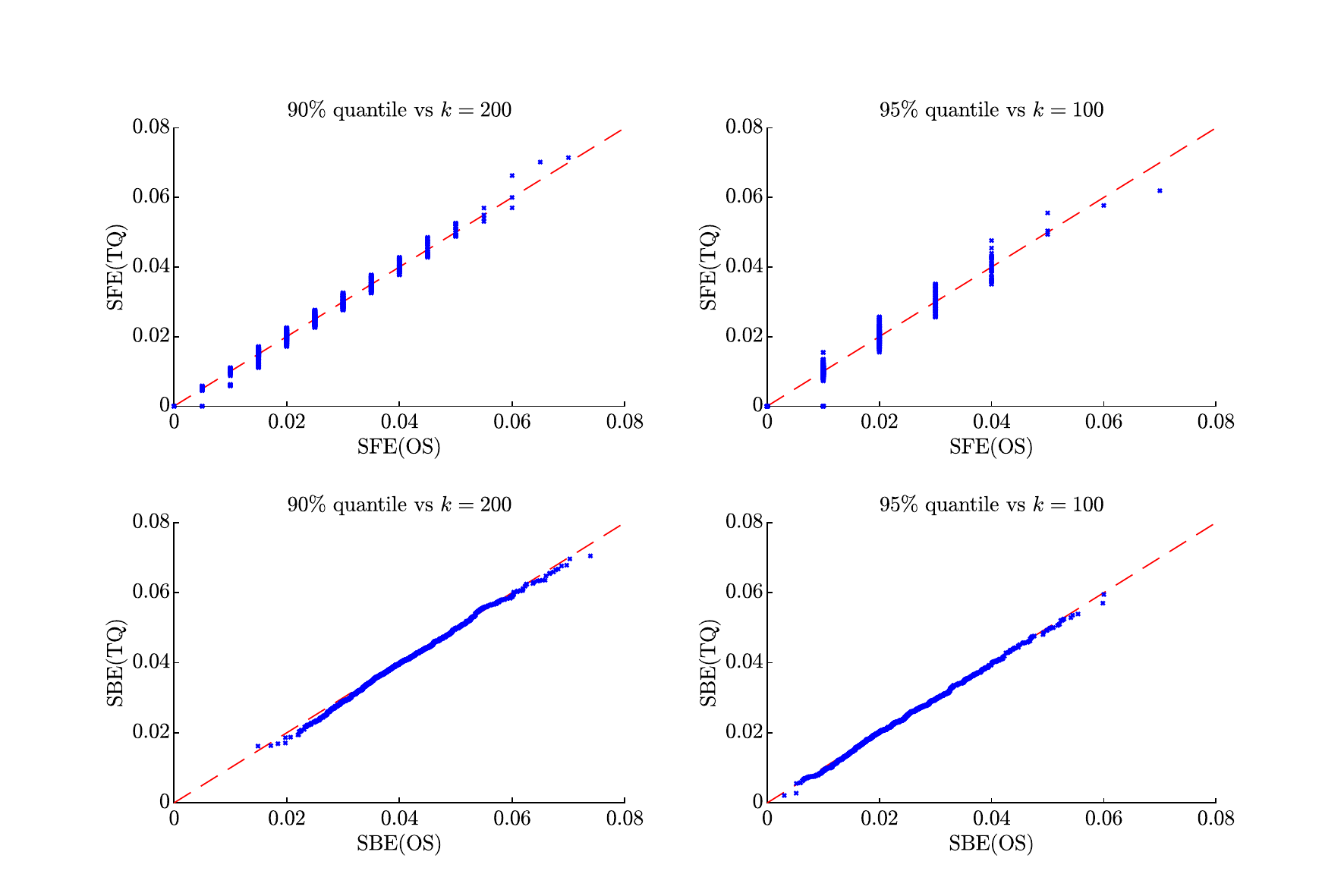}
	\caption{Comparison of estimators of  $P\{\Theta_5>1\}$  in the tCopula model  with $\rho=0.25$}
	\label{figtcop1}
\end{figure}
\begin{figure}[tb]
	\centering
	\includegraphics[width=1\textwidth]{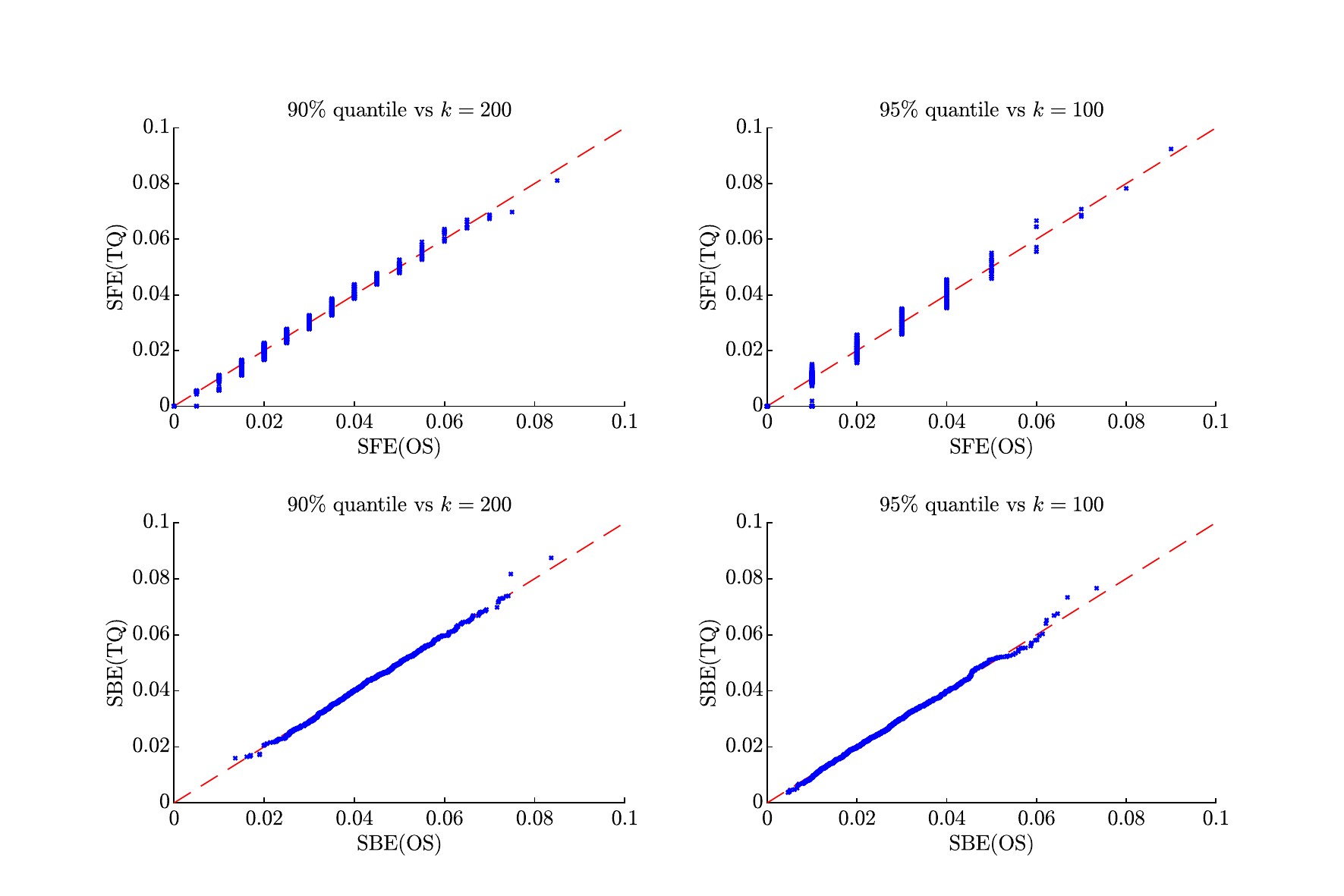}
	\caption{Comparison of estimators of  $P\{\Theta_5>1\}$  in the tCopula model  with $\rho=0.5$}
	\label{figtcop2}
\end{figure}
\begin{figure}[tb]
	\centering
	\includegraphics[width=1\textwidth]{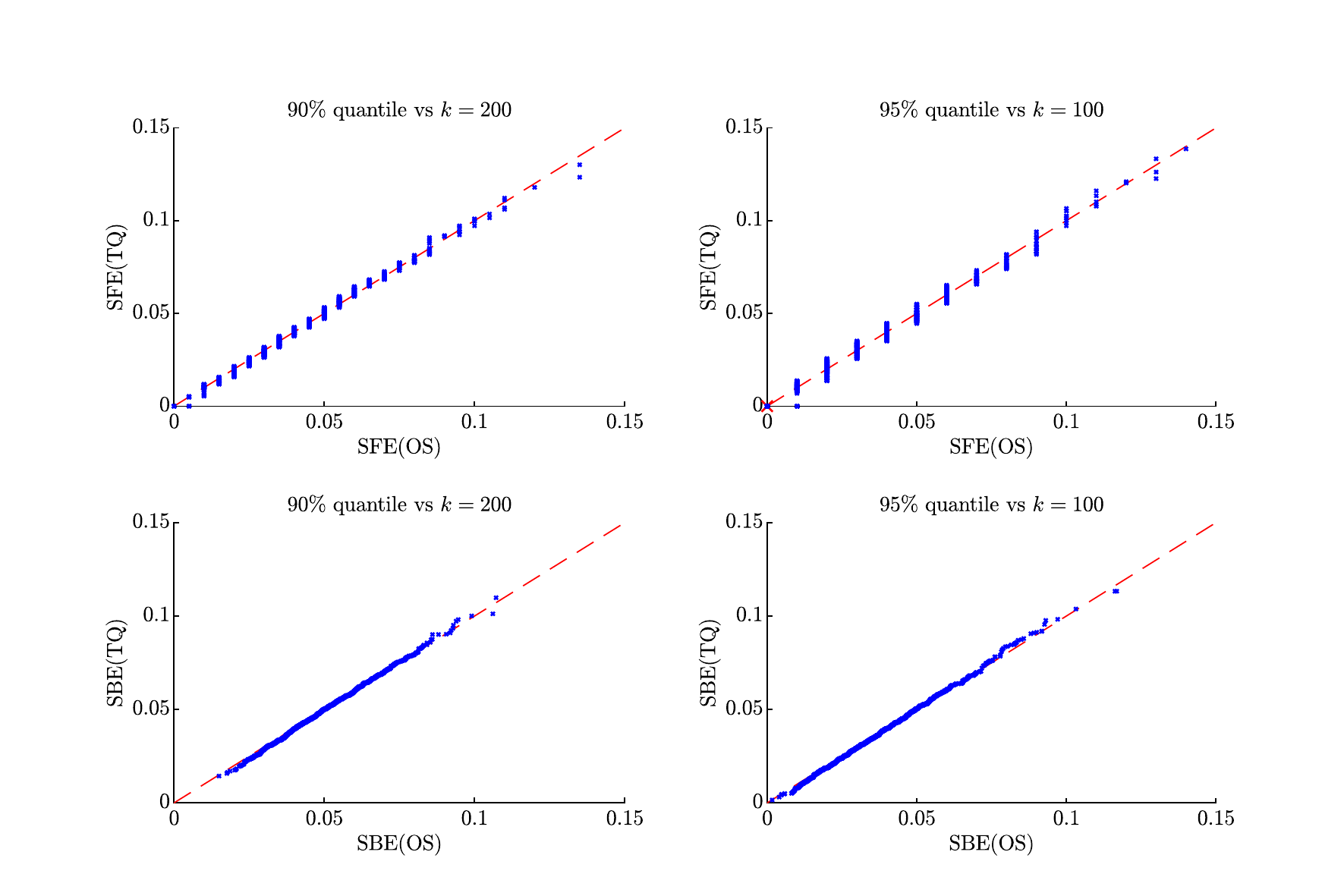}
	\caption{Comparison of estimators of  $P\{\Theta_5>1\}$  in the tCopula model  with $\rho=0.75$}
	\label{figtcop3}
\end{figure}

Finally, we consider the gumCopula model.
Figures \ref{figgum1}--\ref{figgum3} show the Q-Q plots for the estimators of $P\{\Theta_1>1\}$ for  model parameter $\theta=1.2$, $1.5$ and $2$, respectively; so again the serial dependence is increasing. Moreover, Figures \ref{figgum4} and \ref{figgum5} show the results for $\theta=2$ and the probabilities $P\{\Theta_1>1/2\}$ resp.\ $P\{\Theta_5>1/2\}$. Overall, the results are similar as for the tCopula model. However, in some settings (in particular for lag 5 and a high threshold) the points lie below the diagonal in the right tail, i.e.\ here the TQ-version performs sometimes slightly better than the OS-version.

\begin{figure}[tb]
	\centering
	\includegraphics[width=1\textwidth]{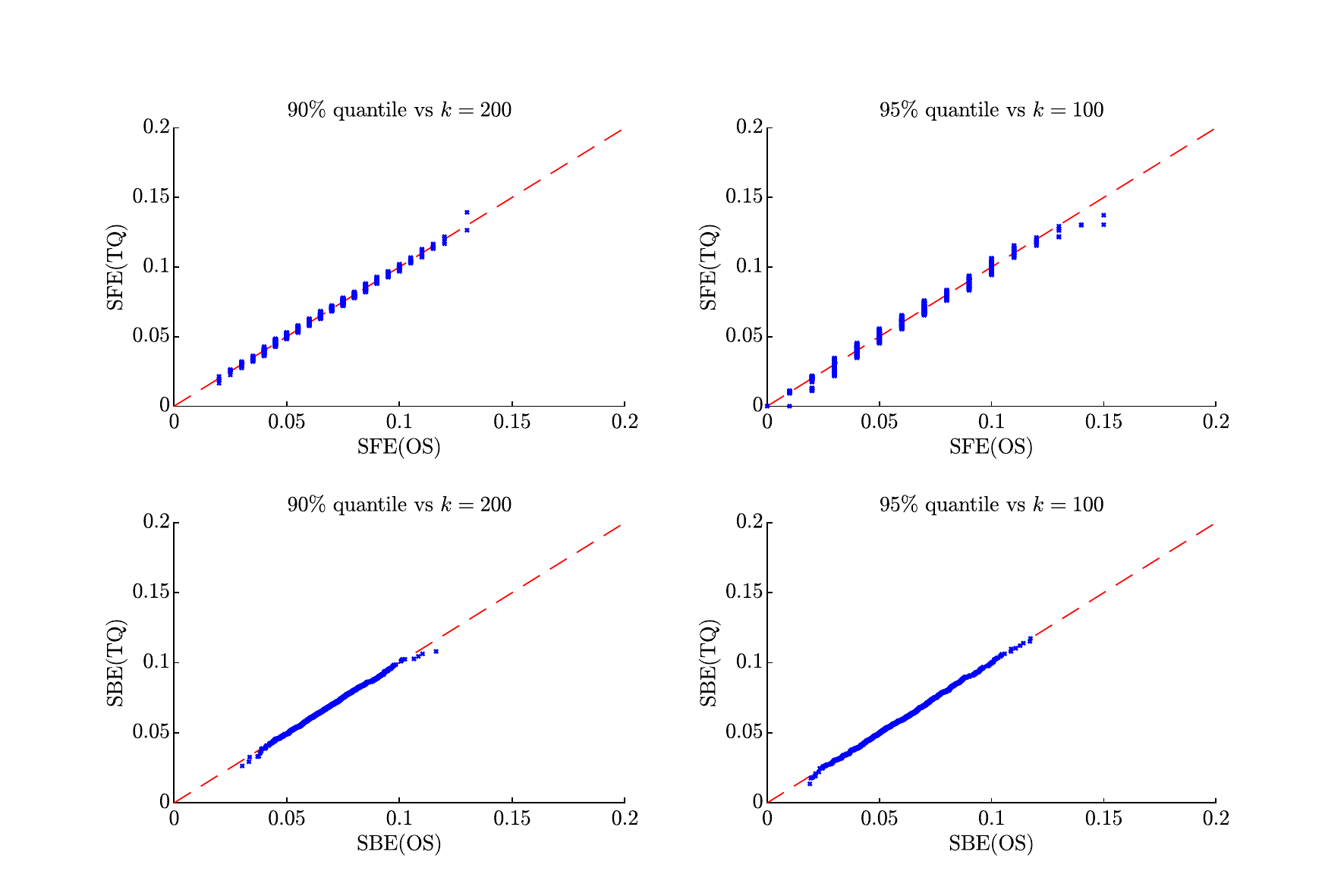}
	\caption{Comparison of estimators of  $P\{\Theta_1>1\}$  in the gumCopula model  with $\theta=1.2$}
	\label{figgum1}
\end{figure}

\begin{figure}[tb]
	\centering
	\includegraphics[width=1\textwidth]{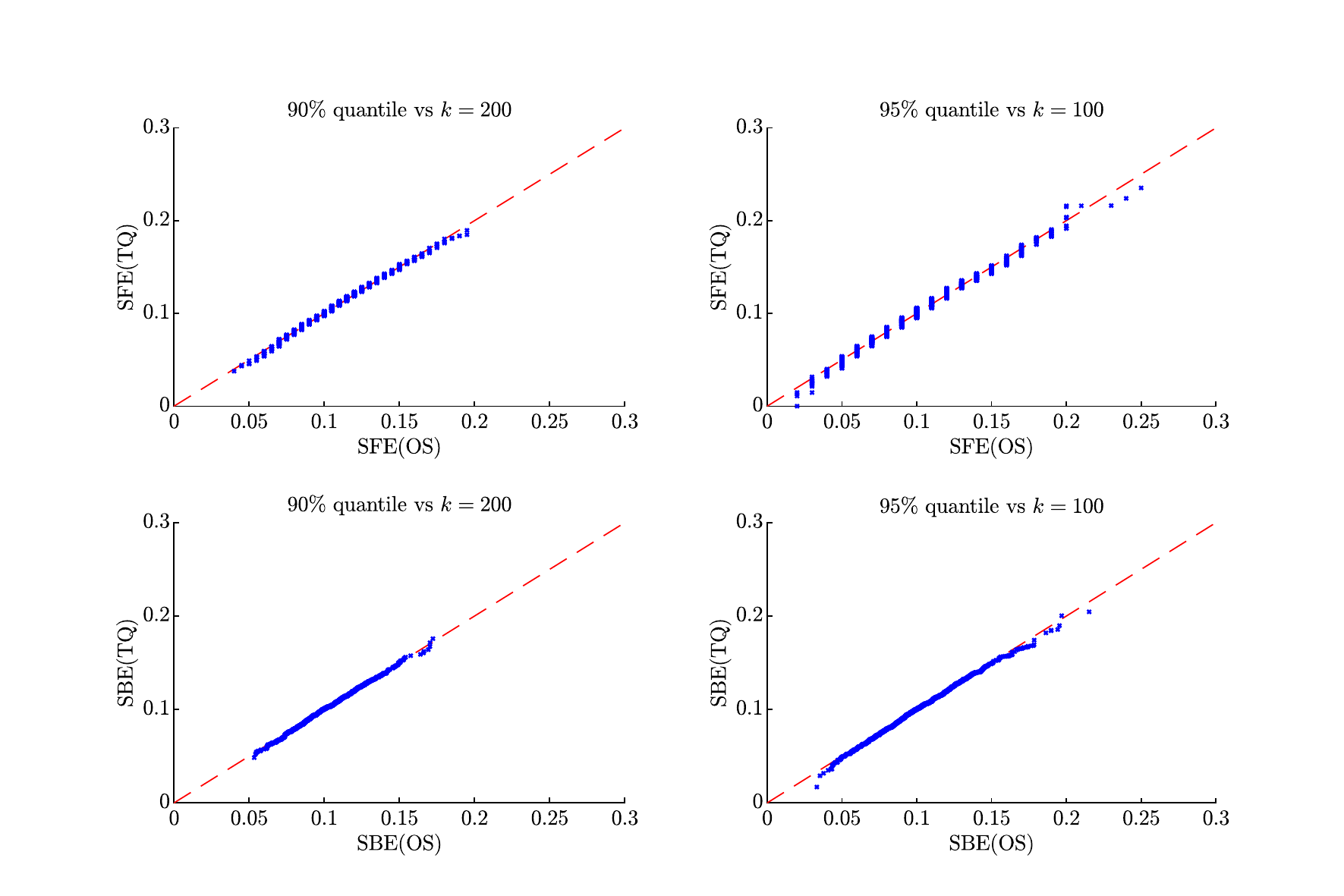}
	\caption{Comparison of estimators of  $P\{\Theta_1>1\}$  in the gumCopula model  with $\theta=1.5$}
	\label{figgum2}
\end{figure}

\begin{figure}[tb]
	\centering
	\includegraphics[width=1\textwidth]{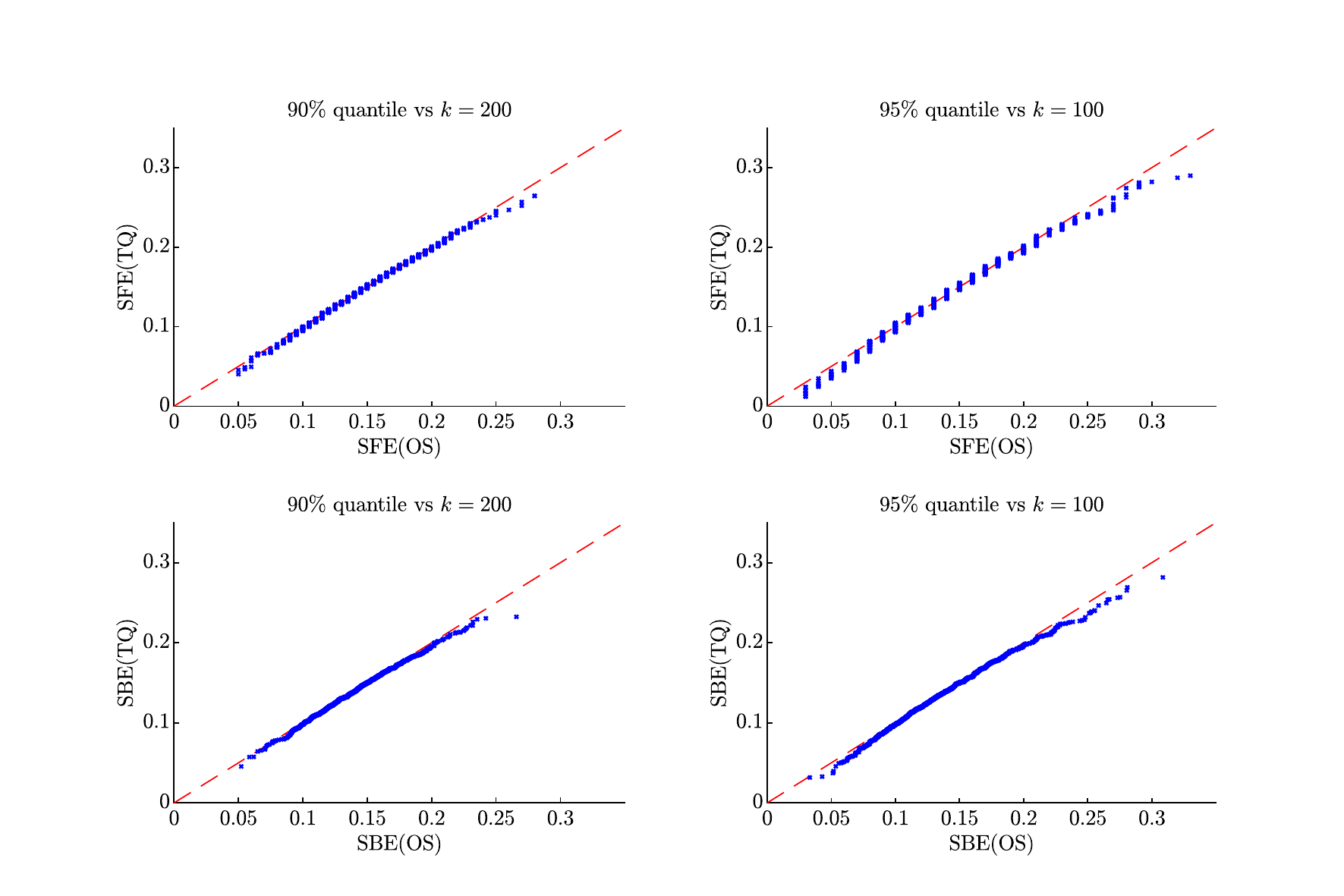}
	\caption{Comparison of estimators of  $P\{\Theta_1>1\}$  in the gumCopula model  with $\theta=2$}
	\label{figgum3}
\end{figure}

\begin{figure}[tb]
	\centering
	\includegraphics[width=1\textwidth]{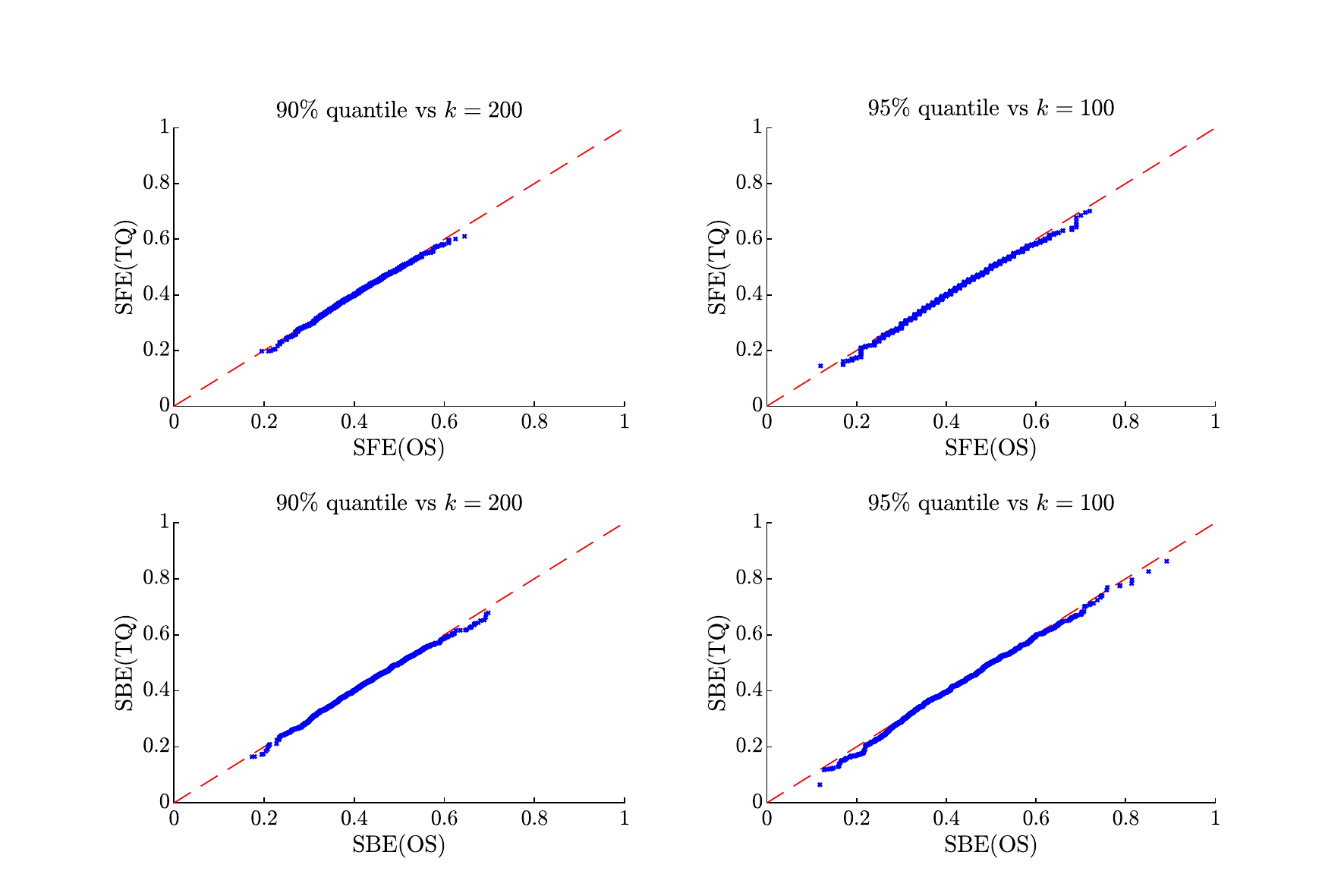}
	\caption{Comparison of estimators of  $P\{\Theta_1>1/2\}$  in the gumCopula model  with $\theta=2$}
	\label{figgum4}
\end{figure}
\begin{figure}[tb]
	\centering
	\includegraphics[width=1\textwidth]{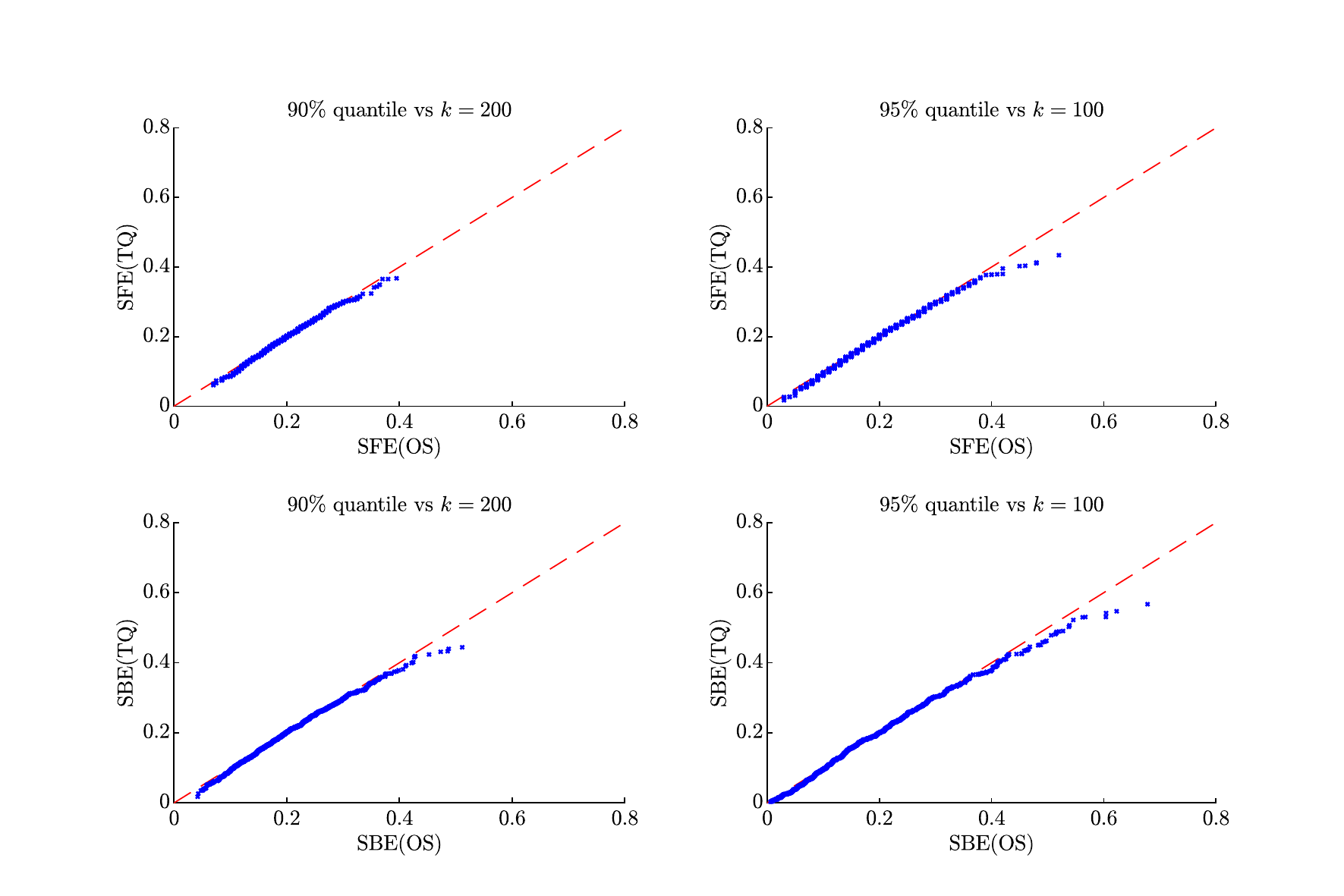}
	\caption{Comparison of estimators of  $P\{\Theta_5>1/2\}$  in the gumCopula model  with $\theta=2$}
	\label{figgum5}
\end{figure}

To sum up, for all constellations under consideration the distribution of both versions of the forward and the backward estimators, based either on exceedances over true quantiles or over order statistics, are similar. Often, the version using true quantiles tends to underestimate the true value more strongly than the version based on order statistics when both estimators yield too small values, whereas the findings are mixed if both versions overestimate the true values.



%
%